\documentclass[11pt]{article}    % Specifies the document style.
\usepackage{amsmath, amssymb, amsthm,pdfsync}
\usepackage{fullpage}
\usepackage{verbatim}
\usepackage{graphicx}
\usepackage{subfig, caption}
\usepackage{cancel}
\usepackage{enumitem}
\usepackage{wrapfig}
\usepackage{mathtools}
\usepackage{xcolor}
\usepackage{epigraph}
\usepackage{cleveref}
\usepackage{overpic}
\usepackage{esint}
\usepackage{dsfont}
\usepackage{color}
\usepackage{titling}
\usepackage{soul}

\makeatletter
\newcommand\ReDeclareMathOperator{%
  \@ifstar{\def\rmo@s{m}\rmo@redeclare}{\def\rmo@s{o}\rmo@redeclare}%
}
\newcommand\rmo@redeclare[2]{%
  \begingroup \escapechar\m@ne\xdef\@gtempa{{\string#1}}\endgroup
  \expandafter\@ifundefined\@gtempa
     {\@latex@error{\noexpand#1undefined}\@ehc}%
     \relax
  \expandafter\rmo@declmathop\rmo@s{#1}{#2}}
\newcommand\rmo@declmathop[3]{%
  \DeclareRobustCommand{#2}{\qopname\newmcodes@#1{#3}}%
}
\@onlypreamble\ReDeclareMathOperator
\makeatother

\setlist[itemize]{itemsep=-1mm}

\newtheorem{theorem}{Theorem}
\newtheorem{lemma}[theorem]{Lemma}
\newtheorem{lem}[theorem]{Lemma}
\newtheorem{prop}[theorem]{Proposition}

\newcommand{\vertiii}[1]{{\left\vert\kern-0.25ex\left\vert\kern-0.25ex\left\vert #1 
    \right\vert\kern-0.25ex\right\vert\kern-0.25ex\right\vert}}

\newcommand{\R}{\mathbb{R}}

\newcommand{\1}{\mathds{1}}

\newcommand{\e}{\epsilon}

\DeclareMathOperator{\Int}{Int}

\ReDeclareMathOperator{\Re}{Re}
\ReDeclareMathOperator{\Im}{Im}

\numberwithin{equation}{section}
\numberwithin{theorem}{section}
\def\beq{\begin{equation}}
\def\eeq{\end{equation}}

\Crefname{assumption}{Assumption}{Assumptions}
\Crefname{theorem}{Theorem}{Theorems}
\Crefname{lem}{Lemma}{Lemmas}
\Crefname{cor}{Corollary}{Corollaries}
\Crefname{prop}{Proposition}{Propositions}
\Crefname{theorem}{Theorem}{Theorems}
\Crefname{conjecture}{Conjecture}{Conjectures}

{\setlength{\parindent}{0cm}

\begin{document}

%\newlength{\ipjwidth}\setlength{\ipjwidth}{339.66878pt}\addtolength{\ipjwidth}{-\linewidth}\typeout{\the\ipjwidth}
\title{The reactive-telegraph equation and a related kinetic model}

\setlength\thanksmarkwidth{.5em}
\setlength\thanksmargin{-\thanksmarkwidth}

\author{
	Christopher Henderson\thanks{Corresponding author, Department of Mathematics, The University of Chicago, 5734 S.~University Avenue, Chicago, IL 60637, E-mail: \texttt{henderson@math.uchicago.edu}} 
\ and	
	Panagiotis E.~Souganidis\thanks{Department of Mathematics, The University of Chicago, 5734 S.~University Avenue, Chicago, IL 60637, E-mail: \texttt{souganidis@math.uchicago.edu}}
}

\maketitle
\begin{abstract}
\noindent We study  the long-range, long-time behavior of the reactive-telegraph equation and a related reactive-kinetic model.  The two problems  are equivalent in one spatial dimension.  We point out that the reactive-telegraph equation, meant to model a population density, does not preserve positivity in higher dimensions.  In view of this, in dimensions larger than one,  we consider a reactive-kinetic model and investigate the long-range, long-time limit of the solutions. We provide a general characterization of the speed of propagation and we compute it explicitly in one and two dimensions. We show that a phase transition between parabolic and hyperbolic behavior takes place only in one dimension. Finally, we investigate the hydrodynamic limit of the limiting problem.
%Hamilton-Jacobi equation arising after the long-range, long-time limit.
\end{abstract}
%\noindent{\bf Key-Words:} {Structured populations, Reaction-diffusion equations, Logarithmic delay, Front propagation}\\
%\noindent{\bf AMS Class. No:} {35Q92, 45K05, 35C07}
%
%\keywords{Navier-Stokes equation, infinite energy solutions, extended system, long-time behavior, Lyapunov function, asymptotic stability}
%\subjclass{76D05, 35Q30, 76M25, 65M06}

%\tableofcontents
%\section{Introduction and Main Results}\label{sec:results}
%\begin{document}

\section{Introduction}

The Fisher-KPP equation
\[
	u_t = u_{xx} + u(1-u)
\]
is the classical model used to study the spread of a population in an environment. %~\cite{Fisher, KPP}.  
Being based on the heat equation, the Fisher-KPP equation exhibits infinite speed of propagation.  Indeed  the population density is non-zero everywhere for any positive time even if the initial data is compactly supported.  
\smallskip

One approach to remove this unphysical behavior is to look at the reactive-telegraph equation
\begin{equation}\label{eq:reactive_telegraph}
	\tau \rho_{tt} + (1 - \tau (1-2\rho)) \rho_t = \Delta \rho + \rho(1-\rho).
\end{equation}
This may be stated with more general non-linearities $F(\rho)$ in place of $\rho(1-\rho)$ and $F'(\rho)$ in place of $1-2\rho$.  In one dimension, this is equivalent, for $\rho=\frac{p^-+p^+}{2}$, to the kinetic system
\begin{equation}\label{eq:kinetic_system1}
\begin{cases}
	p_t^+ + \frac{1}{\sqrt \tau} p_x^+ = \frac{1}{2\tau} \left( p^- - p^+\right) + \frac{1}{2} \rho(1-\rho),\\
	p_t^- - \frac{1}{\sqrt \tau} p_x^- = \frac{1}{2\tau} \left( p^+ - p^-\right) + \frac{1}{2} \rho(1-\rho).
\end{cases}
\end{equation}
The solution $\rho$ is a population density, while $p^\pm$ represents the density of those individuals moving with velocity $\pm 1$.  The positive parameter $\tau$ is related to  the relaxation time, which depends  on  the mean run length between changes in velocity.  A small sample of the more applied literature addressing these models, including the interpretation of $\tau$ can be found in the works of Fedotov \cite{Fedotov}, Hillen and Othmer \cite{HillenOthmer}, Holmes \cite{Holmes}, Horsthemke \cite{Horsthemke}, Kac \cite{Kac}, M\'endez, Fort, and Farjas \cite{mendez1999speed}, and Ortega, Fort, and M\'endez \cite{ortega2004role}. %We point the interested reader to  Fedotov~\cite{Fedotov}, Kac~\cite{Kac}, %Hillen~\cite{Hillen},
% Hillen and Othmer~\cite{HillenOthmer}, and Holmes~\cite{Holmes}.
%\smallskip
\vskip.075in

We discuss, in particular the work in~\cite{Fedotov}, where the author proposed~\eqref{eq:reactive_telegraph} as a model for population dynamics in more than one dimension, deriving it from a general transport model with the flux given by convolution of the gradient with a kernel, and formally analyzed the long-range, long-time behavior of the solution  with arguments similar to those used to rigorously study the same problem for the Fisher-KPP equation.  However, to even write down this formal analysis it is necessary to know that the solutions to  \eqref{eq:reactive_telegraph} preserve the sign, that is if they start nonnegative, they remain nonnegative. 
\smallskip

Our first result is to show that, in general, such formal computations cannot be justified when $n\geq 2$.  Indeed, solutions to~\eqref{eq:reactive_telegraph} need not remain positive.  In other words, it is not possible to  rule out negative population densities, which are unphysical.
\begin{theorem}\label{thm:non_positive_telegraph}
Assume $n = 1$ and fix  $\rho_0 \in C^1(\R)\cap L^1(\R)$ with $0 \leq \rho_0 \leq 1$.  If  $\rho$ is the solution of ~\eqref{eq:reactive_telegraph} with initial data $\rho(\cdot,0) = \rho_0$ and $\rho_t(\cdot,0) = 0$, then $0 \leq \rho(x,t) \leq 2$ for all $(x,t) \in \R\times [0,\infty)$.
\smallskip

If $n\geq 2$, there exists $\rho_0 \in C^\infty(\R^n)$ with $0 \leq \rho_0 \leq 1$ and $(x_0,t_0) \in \R^n\times \R_+$  such that,  if $\rho$ solves  ~\eqref{eq:reactive_telegraph} with initial data $\rho(\cdot,0) = \rho_0$ and $\rho_t(\cdot,0)=0$,  then  $\rho(x_0, t_0) < 0$.
\end{theorem}
%We believe that~\Cref{thm:non_positive_telegraph} holds for $n=2$, but, in the interest of simplicity, we only consider the higher dimensional case.  Indeed, we reduce the general case to considering only $n=3$, where our proof relies on the Duhamel formula along with the explicit form of the solution to the wave equation. 
%  This formula is significantly simpler in three dimensions; however, we believe that a similar example could be constructed in two dimensions as well.
%\smallskip
\smallskip

We discuss first why~\eqref{eq:reactive_telegraph} preserves positivity when $n=1$.   In this case,~\eqref{eq:reactive_telegraph} is equivalent to~\eqref{eq:kinetic_system1}, which 
% and we are able to show that~\eqref{eq:the_prelim_equation} 
 preserves positivity because, roughly, %with $F$ given by either~\eqref{takis210} or ~\eqref{assumption:reaction} because, in one dimension, %~\eqref{eq:the_prelim_equation} preserves positivity 
  $\rho$ is controlled by $p$ from below.  Neither the equivalence of the model~\eqref{eq:reactive_telegraph} to a kinetic equation nor the fact that $\rho$ is controlled below by $p$ holds in higher dimensions.
 \smallskip

As is evident in the proof of \Cref{thm:non_positive_telegraph},~\eqref{eq:reactive_telegraph} does not preserve positivity for a wide class of non-linearities $F$.  The essential ingredients of the proof are that the equation is well-defined for short times for data in a ``nice'' enough Sobolev space and that $F$ is piecewise $C^1$ near zero with $F(0) = 0$.  In particular, the second claim in~\Cref{thm:non_positive_telegraph} is due to the properties of the wave operator and not the non-linearity $\rho(1-\rho)$.
\smallskip

In view of the discussion above, the reactive-telegraph equation allows for negative population densities, suggesting that it may not be an appropriate biological model in some settings.  As such, we restrict our focus to the following  generalization of the kinetic system~\eqref{eq:kinetic_system1} 
%\begin{equation}\label{eq:the_prelim_equation}
%\begin{cases}
%	p_t + a_{n,\tau}v \cdot D p = \frac{1}{\tau} \left( \rho - p\right) + F(\rho) \ \text{ in }\  \R^n\times S^{n-1} \times \R_+,\\[1mm]
%	p(x,v,0) = p_0(x,v), \ \text{ in }  \   \R^n \times S^{n-1},
%\end{cases}
%\end{equation}
\begin{equation}\label{eq:the_prelim_equation}
p_t + a_{n,\tau} v \cdot D p = \frac{1}{\tau} \left( \rho - p\right) + \rho(1-\rho)_+ \ \text{ in }\  \R^n\times S^{n-1}\times \R_+,
\end{equation}
where $D$ is the spatial gradient, $p$ is the density of individuals moving with velocity $v \in S^{n-1}$,
$$\rho(x,t): = \fint_{S^{n-1}} p(x,v,t)dv$$ 
is the population density, 
\beq\label{takis230}
a_{n,\tau} := \sqrt{n/\tau}
\eeq
 is the speed of pure transport, $S^{n-1}$ is the unit sphere in $\R^n$, $\R_+:=(0,\infty),$ 
%it is, of course, immediate that, when $n=1$, \eqref{eq:the_prelim_equation} reduces to \eqref{eq:kinetic_system1}.  
$\fint$ denotes the normalized integral such that $\fint_{S^{n-1}}dx =1$, and $x_+ := \max(x,0)$.
\smallskip

We make a few general comments regarding~\eqref{eq:the_prelim_equation}.  First, although the reactive-telegraph equation \eqref{eq:reactive_telegraph} is equivalent to a kinetic model like~\eqref{eq:the_prelim_equation} (cf.~\eqref{eq:kinetic_system1}) in one dimension, there is no rigorous connection between the two in higher dimensions.  Second, we use the non-linearity $\rho(1-\rho)_+$ so that the model preserves positivity\footnote{See \Cref{lem:a_priori} for a proof of the preservation of positivity and see \Cref{sec:positivity} for a discussion of a related model with the logistic non-linearity $\rho(1-\rho)$ that does not preserve positivity, which suggests that $\rho(1-\rho)_+$ is a better choice of non-linearity.} and retains the fundamental aspects of the logistic one, $\rho(1-\rho)$. In particular, it represents the physical assumption that growth and competition depend only on the total population at a particular location and do not depend on velocity. 
Lastly, the $\sqrt n$ factor in $a_{n,\tau}$ is to fix the hydrodynamic limit $\tau \to 0$ as the classical Fisher-KPP equation regardless of the dimension.  This is discussed  later in the paper (see \Cref{prop:tau}).
\smallskip

By analogy with the Fisher-KPP equation, we expect the population to spread linearly-in-time.  In situations like this, where a front\footnote{By ``front,'' we mean the area between where $\rho$ approximately takes the value $0$ and the value $1$.} is expected to move at an approximately constant speed, it is standard to use the hyperbolic long-range, long-time limit, see Barles, Evans and Souganidis~\cite{BarlesEvansSouganidis}, Evans and Souganidis~\cite{EvansSouganidis}, and  Majda and Souganidis~\cite{MajdaSouganidis} and the large body of literature citing these works. This limit corresponds to scaling space and time by the same large parameter in order to capture this linear-in-time propagation while ignoring fluctuations and short-time behavior.
\smallskip

As such, we use the hyperbolic scaling $(x,v,t) \mapsto (x/\e,v,  t/\e)$ and consider  the rescaled function $$p^\e(x,v,t) := p(x/\e,v,t/\e),$$ which solves 
\begin{equation}\label{takis1}
p^\e_t + a_{n,\tau} v \cdot D p^\e = \frac{1}{\e\tau} \left( \rho^\e - p^\e\right) +\frac{1}{\e}  \rho^\e(1-\rho^\e)_+ \ \ \text{in} \ \ \R^n\times S^{n-1}\times \R_+,
\end{equation}
with $$\rho^\e(x,t): = \fint_{S^{n-1}} p^\e(x,v,t)dv.$$
\smallskip

We study the behavior, as $\e\to 0$, of $p^\e$ and $\rho^\e$ with initial datum %assume next that 
\begin{equation}\label{takis3}
p^\e(\cdot,\cdot,0) = p_0 \  \text{ on } \ \R^n\times  S^{n-1} \times \{0\}
\end{equation}
%Let $G_0: = \{x \in \R^n : p_0^\e(x,v)>0 \ \text{for all} \ v\in S^{n-1}\}.$
%\begin{assumption}\label{assumption:G_0}
%such that \textbf{there exists $G_0 \subset \R^n$ where}%We assume that $p_0^\e$ and the compact set $G_0 \subset \R^n$ satisfy, independent of $\e$,
%\begin{equation}\label{assumption:G_0}
%%\begin{cases}
% 	G_0=\{x\in \R^n:  \inf_v
%	%_{v \in S^{n-1}} 
%	p_0(x,v) >0\}, \ 
% 	G_0^c = \{x \in \R^n: \sup_v p_0(x,v) = 0\}  \text{ and } 0\leq p_0\leq 1.\\[1mm]
%	 % \text{ and }%\\[1mm]
% % \qquad\qquad\qquad  0\leq p_0\leq 1.
%%  \end{cases}
%  \end{equation}
such that
  \begin{equation}\label{assumption:G_0}
\begin{cases}
 0\leq p_0\leq 1, \text{and there exists an open, Lipschitz domain $G_0 \subset \R^n$ so that}\\[1mm]
% \text{there exists $G_0 \subset \R^n$ such that}	
 G_0:=\{x\in \R^n:  \inf_v
	%_{v \in S^{n-1}} 
	p_0(x,v) >0\} \  \text{and} \ %\\[1mm] 
 	G_0^c = \{x \in \R^n: \sup_v p_0(x,v) = 0\}. %  \text{ and } 0\leq p_0\leq 1.%\\[1mm]
	 % \text{ and }%\\[1mm]
 % \qquad\qquad\qquad  0\leq p_0\leq 1.
  \end{cases}
  \end{equation}
%We note that the $\sqrt n$ term is to fix the hydro-dynamic limit to be the classical Fisher-KPP equation, regardless of dimension.  This is discussed  later in the paper (see \Cref{prop:tau}).
%\smallskip
%
In order to investigate the propagation properties  of  $p^\e$ and $\rho^\e$ as $\e\to 0$, we use 
the classical Hopf-Cole transform $p^\epsilon = \exp\left(J^\e/\e\right)$, which is a standard tool in the study of front propagation~\cite{BarlesEvansSouganidis, EvansSouganidis, MajdaSouganidis},  and study first the behavior, as $\e \to 0$, of the $J^\e$'s, which solve 
\begin{equation}\label{eq:the_equation}
\begin{cases}
J_t^{\e} + a_{n,\tau} v \cdot D J^\e& = \frac{1}{\tau} \left( \fint_{S^{n-1}} e^{\frac{1}{\e} (J^\e(x,v',t) - J^\e(x,v,t)) }dv' - 1\right) \\[4.5mm]  \qquad & +  \fint_{S^{n-1}} e^{\frac{1}{\e}{(J^\e(x,v',t) - J^\e(x,v,t))}} dv' (1-\rho^\e)_+ \ \text{ in } \  \R^n \times S^{n-1} \times \R_+.
\end{cases}
\end{equation}
Note that it is possible to use this transformation since  $p^\e \geq 0$ in $\R^n\times S^{n-1}\times \R_+$. This is shown to be the case in  Lemma~\ref{takis10} if \eqref{assumption:G_0} holds.
\smallskip

In what follows, when necessary to signify the dependence on $\tau$, we write   $J^{\e,\tau}$, $J^\tau$ and $H^\tau$ instead of $J^\e$, $J$ 
and $H$. 

% and , but, since 
% it is rather cumbersome, we avoid it. We do, however,  write $J$ for the limit.
\begin{theorem}\label{thm:phi_convergence}
Assume \eqref{assumption:G_0} and let $J^\e$ solve~\eqref{eq:the_equation} with initial data $J^\e(\cdot,\cdot,0)=\e \log p_0$ on $\R^n\times  S^{n-1}.$  Then,  for each $\tau >0$, there exists a concave, rotationally invariant  function $H: \R^n \to \R$, defined in~\eqref{takis30},  such that, as $\e \to 0$ and  uniformly in $v$ and locally uniformly in $\R^n \times \R_+$,  $J^\e$ converges to  $J$, the unique solution to
\begin{equation}\label{eq:J}
\max\left (J_t + H(DJ), J \right)= 0\  \text{ in } \ \R^n \times \R_+  \qquad J(\cdot,0) = \begin{cases}
		- \infty \ \ &\text{ in } \ \  \overline G_0^c,\\[1mm]
		0 \ \ &\text{ in } \ \ G_0.
	\end{cases}
%\max\left\{\frac{\tau}{1+\tau} - \fint_{S^{n-1}} \frac{dv}{\frac{1}{\tau} + J_t + \frac{\sqrt n}{\sqrt{\tau}}v \cdot D J}, J \right\} = 0 \  \text{ in } \ \R^n \times \R_+,
\end{equation}
%with  initial condition
%\begin{equation}\label{takis5}
%	J(\cdot,0) = \begin{cases}
%		- \infty \ \ &\text{ in } \ \  \overline G_0^c,\\[1mm]
%		0 \ \ &\text{ in } \ \ G_0.
%	\end{cases}
%	\end{equation}
%notice that $J$ does not depend on $v$.  
%The Hamiltonian $H$ is defined in~\eqref{takis30}. % in \Cref{sec:hamiltonian}.
\end{theorem}

We make a few brief comments on some technical aspects of \Cref{thm:phi_convergence}.  First, the limiting function $J$ does not depend on $v$.  Second, when $p_0$ is zero, we define $\log p_0$ to be the extended real value $-\infty$.  Lastly, by the locally uniform convergence of $J^\e$ to $J$, we mean that this convergence is uniform on compact sets $K$ such that either $K \subset \Int\{J = -\infty\}$ or $K \subset \Int\{J < -\infty\}$.
\smallskip

Knowing Theorem~\ref{thm:phi_convergence} we then infer the following spreading behavior of  $\rho^\e$. % from this.
\begin{theorem}\label{thm:propagation}
	Suppose that \eqref{assumption:G_0}   holds and that $J$ is the solution to~\eqref{eq:J}. Then, as $\e\to 0$ and locally uniformly in $\{J < 0\}$,  $\lim_{\e \to 0} \rho^\e(x,t)= 0.$
%%	 and let $p^\e$ satisfy~\eqref{eq:the_prelim_equation} and \eqref{takis3}. Then, as $\e\to 0$ and locally uniformly in $(x,t)$,
%	\[\lim_{\e \to 0} \rho^\e(x,t)= 0.\]% \ \text{ in } \  \{J <0\}.\]
%	
\smallskip

If $\tau\leq 1$, then, 	 as $\e\to 0$ and locally uniformly in $\Int \{ J = 0\},$ $\lim_{\e \to 0} \rho^\e(x,t)=1.$
%\[\lim_{\e \to 0} \rho^\e(x,t)=1.\] % \ \text{ in } \ \Int \{ J = 0\}.\]
\smallskip

If $\tau >1$, then, as $\e\to0$ and locally uniformly in $\Int \{ J = 0\}, \   \liminf_{\e\to 0} \rho^\e(x,t) \geq 1.$
%\[ \liminf_{\e\to 0} \rho^\e(x,t) \geq 1.\] 
	
%	 \ \text{ in } \ \Int \{ J = 0\}.\]
%		\lim_{\e \to 0} \rho^\e(x,v,t)
%			=\begin{cases}
%				1 \ \text{ in } \ \Int \{ J = 0\},\\[1mm]
%				0 \ \text{ in } \  \{J <0\}.
%			\end{cases}				
%	\]
\end{theorem}

Before discussing the proof of \Cref{thm:propagation}, we mention the reason that there is a distinction between $\tau >1$ and $\tau \leq 1$.  When $\tau<1$, it is possible to bound $p$ by 1 by using a maximum principle-type argument.  On the other hand, when $\tau \geq 1$, it is shown in \Cref{lem:a_priori} that   upper bound is order  $(1+\tau)^2/4\tau$. The proof of \Cref{thm:propagation} yields that $\liminf_{\e\to0} \rho^\e(x,t) \geq 1$ on $\Int\{J=0\}$.  When $\tau <1$, the bound above gives us immediately that $\limsup_{\e \to 0} \rho^\e(x,t) \leq 1$, yielding the exact value of the limit.  When $\tau\geq 1$, this argument does not work.
\smallskip

To prove the asymptotic results, we use the classical half-relaxed limits of Barles and Perthame~\cite{BarlesPerthame} along the lines of~\cite{BarlesEvansSouganidis, EvansSouganidis} and a modification due to Barles and Souganidis~\cite{BarlesSouganidis}, which allows us to side-step the technical difficulty that, due to the finite speed of propagation in kinetic equations, the $J^\e$'s take the value $-\infty$.  The combination of the half-relaxed limits and the technique of~\cite{BarlesSouganidis} is that, roughly, when the limiting Hamilton-Jacobi equation represents a distance function, no a priori bounds or regularity are needed to pass to the limit $\e \to 0$.
\smallskip

The Hamiltonian $H$ in \Cref{thm:phi_convergence} is the same one found by Bouin and Calvez~\cite{BouinCalvezEikonal} and Caillerie~\cite{Caillerie} in a related context since the linearized equations are the same.
% due to the finite speed of propagation in kinetic equations.
\smallskip

%Since the  map $$a \to G(a,p):=-\fint_{S^{n-1}} \frac{dv}{a+ \frac{\sqrt n}{\sqrt{\tau}}v \cdot p}$$ is increasing in $a$,  it is possible to characterize the propagation in another way.   For $p\in \R^n$, let  $H(p)$ be defined by the implicit formula
%$$
%\frac{\tau}{1+\tau} + G(H(p),p)=0. 
%$$
%%\[
%%	\frac{\tau}{1+\tau} = \fint_{S^{n-1}} \frac{dv}{H(p) + v\cdot p}.
%%\]
%It is then follows from elementary considerations from the theory of viscosity solutions that \eqref{eq:J} is equivalent to 
%%clear that $J$ satisfies the equation
%\[
%	\max\left\{ J_t + \tau^{-1} - H (D J), J\right\} = 0  \  \text{ in } \ \R^n \times \R_+.
%\]

Since $H$ in \eqref{eq:J} is space-time homogeneous, %that is, independent of time and  space, 
it follows  from~\cite{EvansSouganidis,MajdaSouganidis} that
\begin{equation}\label{eq:I_tau}
J(x,t)=\min\left( I(x,t), 0\right),
\end{equation}
where $I$ is the solution to 
\[ 
I_t + H(DI) =0 \ \ \text{in} \ \  \R^n\times \R_+,
	 \qquad I(\cdot, 0)= \begin{cases}
		- \infty \ \text{ in } \  \overline G_0^c,\\[1mm]
		0 \ \text{ in} \ G_0.
	\end{cases}
\]
Let $L$ be the concave dual of $H$. Then, see, for example, Lions \cite{LionsBook},
\begin{equation}\label{chris35}
	I(x,t)= t \sup\left\{ L \left( \frac{x-y}{t}
\right) : y\in G_0 \right\}.
\end{equation}

As mentioned in \Cref{thm:phi_convergence}, $H$ is concave and rotationally invarant.  From this it follows that $H$ is radially decreasing; this can  also be seen  from the explicit formula~\eqref{takis30}. As a result, $L$ has the same properties. %As it can be seen from~\eqref{takis30}, $H$ is isotropic and radially decreasing and, as a result, $L$ has the same properties.
Taking some liberty with notation,  we write $L(q)=L(|q|)$. It then follows that, for $t\in \R_+$,
\begin{equation}\label{chris105}
 \{x \in \R^n: J(x,t)<0\} =\left\{ x\in \R^n:   \sup_{y\in G_0} L\left( \frac{x-y}{t}\right)  <0\right\}= \left\{ x\in \R^n:  L\left(\frac{d(x, G_0)}{t}\right)<0 \right\},
\end{equation}
where $d(x,G_0)$ is the usual distance from the point $x$ to the closed set $G_0$. 
\smallskip

In view of \Cref{thm:propagation}, it is clear that the front is $\partial \{x \in \R^n: J(x,t) <0\}$.  From the discussion above, we see that %Hence the front is 
\begin{equation}\label{chris36}
	\partial  \{x \in \R^n: J(x,t)<0\}=\partial \{x\in \R^n: I(x,t)<0\}=%\partial {
	\left\{ x\in \R^n: L\left(\frac{d(x, G_0)}{t}\right)=0 \right\}.
\end{equation}

%= \{x \in \R^n: <0\}$
%\]
%which then
%
%
% is homogeneous, that is, it does not depend on time or space, the solution $J$  Then  arguments in \cite{EvansSouganidis}, Lions~\cite{Lions}, and Majda and Souganidis~\cite{MajdaSouganidis} yield that, if \textbf{Takis -- verify that you're okay with what I wrote here.}
%\[
%	c^n_\tau: = \min_{p_\in \R^n}\left(\frac{-\frac{1}{\tau} - H(p)}{|p|}\right),
%\]
%then
%\[
%	d(G_0,\partial \{x\in \R^n: J(\cdot,t) <0 \})= c^n_\tau t.
%\]
%Indeed, one may see this by working directly with the Lagrangian associated to $H$ and noting that solutions Euler-Langrange equations move along lines at a constant speed.  

The next result, which holds  for  $n\geq 2$, provides a characterization of the $0$-level set of $L$ and, hence, the speed of propagation in terms of a global property of the Hamiltonian.% The claim is true for $n\geq 2$. 
\smallskip

Let 
$$c_{n,\tau}: = - \sup_{q\in\R^n} \frac{H(q)}{|q|}.$$
The claim is: 
% In higher dimensions, the behavior is the same.  Indeed, due to the special form of the Hamiltonian, though we cannot compute the speed explicitly when $n \geq 3$, we can still prove that there is no phase transition between hyperbolic and parabolic behavior by obtaining a useful representation formula for the speed.
 \begin{prop}\label{prop:speedn}
 	Assume~\eqref{assumption:G_0} and  $n \geq 2$.  %$J$ is the solution to~\eqref{eq:J} and~\eqref{takis5}, and $n \geq 2$.  
Then  $c_{n,\tau}$ is achieved,  $c_{n,\tau} < a_{n,\tau},$
% 	\[
% 		c_{n,\tau} = - \max_{p\in\R^n} \frac{H(p)}{|p|} < a_{n,\tau}.
% 	\]
and the  front is the set  $\{x \in \R^n : d(x,G_0) = c_{n,\tau} t\}$.  %\textbf{When $n =1$, if $c_{1,\tau}$ is achieved, the front is the set $\{x \in \R: d(x,G_0) = c_{1,\tau} t\}$.}% where $d$ is the usual distance.  Further $c_{n,\tau} < a_{n,\tau}$.
\end{prop}

Since $c_{n,\tau} < a_{n,\tau},$ it follows that, for $n\geq 2$, the front moves with velocity slower than that of pure transport. Following~\cite{BouinCalvezNadin1}, we call this behavior parabolic  to distinguish it from the hyperbolic one, which is observed, as discussed below, for $n=1$,  when $\tau\geq 1.$
\smallskip

To heuristically justify the term ``parabolic,'' we return to the unscaled problem~\eqref{eq:the_prelim_equation} and discuss the behavior of $\rho$ when $\rho_0$ has compact support.  Indeed, due to the kinetic nature of~\eqref{eq:the_prelim_equation}, the support of $\rho$ propagates with the speed of the pure transport $a_{n,\tau}$. \Cref{prop:speedn}, however, implies %\footnote{Formally, one may see that the propagation is no faster than $c_{n,\tau}$ by scaling the problem~\eqref{eq:the_prelim_equation} with $\epsilon = 1/t$ for $t\gg 1$ and using the locally uniform convergence of \Cref{thm:propagation}.  To make this fully rigorous is straightforward but beyond the scope of this heuristic description.}
that the set on which $\rho$ is approximately $1$ propagates at speed $c_{n,\tau}$. %that the $\lambda$-level sets of $\rho$ propagate at speed $c_{n,\tau}$ for any $\lambda \in (0,1)$. 
This difference suggests that $\rho$ has long tails connecting $0$ and $1$.  As a result the profile resembles that of solutions to the Fisher-KPP equation, whose long tails are caused by the infinite speed of propagation in the heat equation.  See Figure~\ref{fig} for a cartoon picture of this, and see~\cite{BouinCalvezNadin1} for an explicit construction where the long tails are observed.%Indeed, the $0$-level set propagates with the speed of the  pure advection.  The difference between the speed of the front and the speed of the $0$-level set produces long tails connecting $0$ and the front. As a result the profile  resembles the one of solutions to the Fisher-KPP equation, which are caused by the infinite speed of propagation in the heat equation.  This heuristically justifies the term ``parabolic.''
\smallskip

\begin{figure}
\begin{center}
\begin{overpic}[scale=.65]%[grid,tics=10]
		{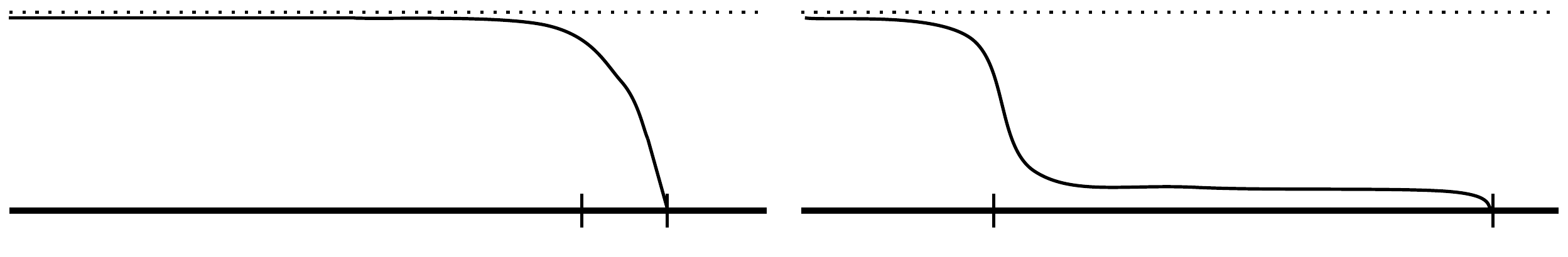}
	\put (34.5,0){$c_{n,\tau}t$}
	\put (40.5,0){$a_{n,\tau}t$}	
	\put (60.5,0){$c_{n,\tau}t$}
	\put (93,0){$a_{n,\tau}t$}
	\put (2,0){$x$}
	\put (0,17) {$1$}
	\put (42, 8) {$\rho$}
	\put (66,8) {$\rho$}
	\put (10,-2.5) {Hyperbolic behavior}
	\put(65,-2.5){Parabolic behavior}
\end{overpic}
\end{center}
\caption{Cartoon pictures showing the difference between the hyperbolic behavior (the graph on the left) seen in one spatial dimension when $\tau \geq 1$ and the parabolic behavior (the graph on the right) seen otherwise.  The key distinction is, in the parabolic behavior, the formation of long tails connecting the boundary of the support of $\rho$, which progress with speed $a_{n,\tau}$, with the location of the front, which progresses with speed $c_{n,\tau}$, which is strictly less than $a_{n,\tau}$.}
\label{fig}
\end{figure}
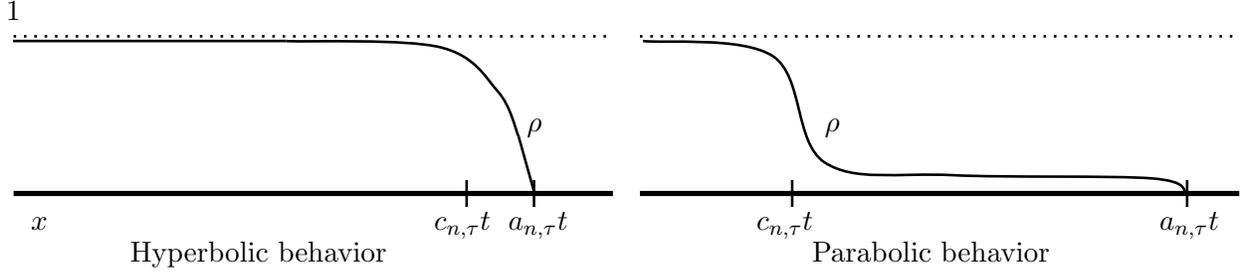

When $n=1$ and $n=2$  it is possible to explicitly compute $H$ and, hence, the speed of propagation.  This is the contents of the proposition below.  We point out that the speed $c_{1,\tau}$ is the one given in~\cite{Fedotov} and is the speed of the traveling waves constructed by Bouin, Calvez, and Nadin in~\cite{BouinCalvezNadin1}.
%A natural question is whether one can characterize the speed with which the population propagates; that is, what is the asymptotic behavior of $\partial\{J = 0\}$.  In dimensions one and two, we compute the speed explicitly.
%n the one and two dimensional cases, we may compute the Hamiltonian and spreading speeds explicitly.  Indeed, we show the following proposition.
\begin{prop}\label{prop:speed12}
	Assume   \eqref{assumption:G_0}. % and  $J$ is the solution to~\eqref{eq:J} and~\eqref{takis5}. 
	Then 
		\[c_{\tau,1} =
		\begin{cases}
			\frac{2}{1+\tau}  \ \text{ if } \ \tau \leq 1,\\[1mm]
			\frac{1}{\sqrt\tau} \ \text{ if } \ \tau \geq 1,
		\end{cases}
\text{and} \quad  c_{\tau,2} = \frac{\sqrt{2(2+\tau)}}{1 + \tau}.
	\]
%	and the front is given by $\{x \in \R^n : d(x,G_0) = c_{n,\tau}t\}$.
%	Then $\partial\{x \in \R^n: I^\tau(x,t) < 0\} = \{x \in \R^n : d(x,G_0) = c_{n,\tau} t\}$, where $d$ is the usual distance.
\end{prop}
When $n=1$ and $\tau \geq 1$, $a_{1,\tau} = c_{1,\tau}$.  We refer to this behavior as ``hyperbolic''.  To heuristically justify this term, we again return to the unscaled problem~\eqref{eq:the_prelim_equation} and consider initial data $\rho_0$ with compact support.  Following the discussion above, we see that the support of $\rho$ and the set on which $\rho$ approximately takes the value $1$ move with the same speed.  In particular, we do not see the formation of long tails as in the discussion above.  Once more, see Figure~\ref{fig} for a cartoon representation of this behavior and see~\cite{BouinCalvezNadin1} for an explicit construction where one sees this behavior.% $0$-level set; that is, the distance between the $0$-level set and the $\lambda$-level set, for any $\lambda \in (0,1)$, remains bounded.  This justifies the term ``hyperbolic'' to describe the behavior in this regime.
\smallskip

It follows from the last two propositions that there is a phase transition, depending on $\tau$, between parabolic and hyperbolic behavior only in one dimension.  We now heuristically explain the reason for this. In order to exhibit hyperbolic behavior, growth due to the reaction must ``keep up'' with pure transport.  That this is possible in only $\R^1$ is related to the fact that $p$ controls $\rho$ from below; that is, for any $v \in S^0=\{-1,1\}$, $p(x,v,t)/2 \leq \rho(x,t)$ for all $(x,t)$.  In particular, the non-local reaction term $\rho(1-\rho)_+$ can be bounded below by a local term when $\rho$ is small. This is not possible in higher dimensions.
\smallskip

Formally, taking the limit $\tau\to 0$ in~\eqref{eq:reactive_telegraph}, we recover the Fisher-KPP equation $u_t = \Delta u + u(1-u)$.  In addition, solutions of~\eqref{eq:the_prelim_equation} formally converge, as $\tau \to 0$, to the Fisher-KPP equation
regardless of dimension;  see, for example, the discussion of Cuesta, Hittmeir, and Schmeiser in~\cite{CuestaHittmeirSchmeiser} and also Section 3 of~\cite{HillenOthmer}.
\smallskip

We prove that the limits may be taken in the reverse order to obtain the same limiting object.
%\smallskip
%
Indeed, fix $u_0$ such that $G_0=\{u_0>0\}$ and let $u^\e$ be the  solution to 
\begin{equation*}
%\begin{cases}
u^\e_t = \e \Delta u^\e + \e^{-1} u^\e(1-u^\e) \ \ \text{in} \ \  \R^n\times \  \R_+ \qquad u^\e(\cdot,0)=u_0 \ \ \text{on} \  \ \R^n. %\times \{0\}%  \\[1mm]
%
%u^\e=u_0 \ \ \text{on} \  \ \R^n\times \{0\}.
%\end{cases}
\end{equation*}
%with $G_0=\{u_0>0\}.$
%\smallskip

It is well-known (see~\cite{EvansSouganidis}) that, as $\e \to 0$ and locally uniformly in $(x,t)$,  $z^\e:=\e \log u^\e$ converges to the unique solution $z$ of
\begin{equation}\label{takis4}
%\begin{cases}	
		\max\{z_t - |Dz|^2 - 1, z\} = 0 \ \ \text{in} \ \ \R^n\times \R_+
		 \qquad  z(\cdot,0)= \begin{cases}
				0 \ \text{ on } \ G_0,\\[1mm]
				-\infty \ \text{ on }  \ \overline G_0^c. % \times \{0\}.
			\end{cases} %\\[1mm]
%z= \begin{cases}
%				0 \ \text{ on } \ G_0 \times \{0\},\\[1mm]
%				-\infty \ \text{ on }  \ \overline G_0^c \times \{0\}.
%			\end{cases}
%\end{cases}
\end{equation}
We show that, as $\tau \to 0$,  $J^\tau \to z$, implying that $\lim_{\tau\to 0}\lim_{\e\to0} J^{\tau,\e} = z$.  This is what one would expect in view of the discussion above, which formally gives that $\lim_{\e\to0}\lim_{\tau\to0} J^{\tau,\e} = z$.%, in view of the discussion in  \cite{CuestaHittmeirSchmeiser}, \textbf{I PROPOSE WE DELETE THIS PARAGRAPH.  I THINK THIS IS WHERE THE REFEREE IS GETTING CONFUSED ABOUT THE LIMIT IN CUESTA, HITTMEIR, AND SCHMEISER BEING FORMAL OR NOT} yields 
%\[
%	\lim_{\tau\to0} \lim_{\e\to0} J^{\tau,\e}
%		= \lim_{\e\to0}\lim_{\tau\to0}  J^{\tau,\e}.
%\]
\begin{prop}\label{prop:tau}
Assume  \eqref{assumption:G_0}  and let  $J^\tau$ and  $z$  be respectively the unique solutions to~\eqref{eq:J} and \eqref{takis4}.  Then, as $\tau \to 0$ and locally uniformly in $\R^n\times \R_+$,  $J^\tau \to z$.
%the unique solution to
%\begin{equation*}
%\begin{cases}	
%		\max\{v_t + |Dv|^2 - 1, v\} = 0 \ \text{in} \ \R^n\times \R_+, \\[1mm]
%v = \begin{cases}
%				0 \ \text{ on } \ G_0 \times \{0\},\\[1mm]
%				\infty \ { on }  \ \overline G_0^c \times \{0\}.
%			\end{cases}
%\end{cases}
%\end{equation*}
\end{prop}
The convergence in \Cref{prop:tau} essentially follows from the formal convergence of $H^\tau$ to $-|\cdot|^2 -1$.
 As seen explicitly in the proof, without the $\sqrt n$ factor in~\eqref{takis230}, there would be a factor $n$ in the equation for $\lim_{\tau\to0} J^\tau$. This is related to the fact that the variance of any unit vector over $S^{n-1}$ is $1/n$.  We mention briefly that, in the equation~\eqref{eq:the_prelim_equation}, the $\sqrt n$ factor can be added or removed by simply scaling in $x$ and therefore does not affect our analysis.% above.  This is seen explicitly in the proof.
\smallskip

Next we discuss briefly related works.  As mentioned above, the one dimensional problems have been considered from the modeling perspective in~\cite{HillenOthmer,Holmes,Kac}.  The reactive-telegraph model~\eqref{eq:reactive_telegraph} has only been rigorously studied in the one dimensional setting.  We point out, in particular, the work of Bouin, Calvez, and Nadin~\cite{BouinCalvezNadin1}, in which they prove existence and stability of traveling waves in a weighted space.  In addition, there  is~\cite{Fedotov} which was discussed earlier. 
% in which he formally computes the propagation speed in the reactive-telegraph equation in all dimensions performing the exponential change of variables. In view of Theorem~\ref{thm:non_positive_telegraph}, however, this is not possible in general.  
There has been more work recently on related kinetic equations in one dimension.  Bouin and Calvez~\cite{BouinCalvezEikonal} obtained  convergence to a limiting equation for the linearization of~\eqref{eq:the_prelim_equation} assuming that $J_0$ is Lipschitz and bounded.  Later Bouin, Calvez, and Nadin~\cite{BouinCalvezNadin2}, and Bouin, Calvez, Grenier, and Nadin~\cite{BCGN} showed acceleration in a kinetic-reactive equation similar to our setting but where the velocity space is unbounded and the reaction term is replaced by $\rho(M(v)-p)$.  In higher dimensions, the kinetic model studied here with $\rho(1-\rho)$ replaced by $\rho(M(v) - p)$ was investigated by Caillerie~\cite{Caillerie}.  In this work Caillerie performs a limiting procedure with well-prepared initial data and identifies the same  Hamiltonian as the one we find above\footnote{That they are related can be seen by the fact that $F(\rho)$ and $\rho(M(v) - p)$ have the same linearization when $M(v)$ is 1 for all $v\in \mathbb{S}^{n-1}$.}.
%\textcolor{red}{{\bf \large{WE SHOULD SAY THAT THEIR NONLINEARITY PRESERVES MAX PRIN }}}
\smallskip

Finally, during the write-up of this work, we became aware of a parallel work by Bouin and Caillerie~\cite{BouinCaillerie} which has  some  similarity  to the present one. In~\cite{BouinCaillerie}, the authors investigate a related kinetic system with reaction term $\rho(M(v)-p)$ and with general dispersal kernel $M(v)$. 
The authors consider the long-range, long-time limit of this equation, the question of propagation in the unscaled equation, and the existence of traveling waves.  There are many differences between the two papers.  First, in the long-range, long-time limit, Bouin and Caillerie consider well-prepared initial data\footnote{The authors fix initial data for their analogue of~$J^\e$ and apparently assume that this initial data is independent of $\epsilon$ and finite everywhere.  This is in contrast to the present work in which $J^\e$ takes infinite values in $G_0^c$.} for $J^\epsilon$, which corresponds to initial conditions that are exponentially decaying (in contrast to the compactly supported initial conditions considered here and widely considered in the front propagation literature\footnote{ Since the study of front propagation is the study of a new population (or other physical quantity) invading a previously unexplored area, which corresponds to $p_0$ being zero, this type of assumption is crucial to a full investigation.  We note that this is not a mere technical detail -- there are many situations where the behavior of compactly supported initial data and well-prepared initial data are quite different (see, e.g., Bramson~\cite{Bramson83}, Zlato\v{s}~\cite{Zlatos}, Cabre and Roquejoffre~\cite{CabreRoquejoffre}).}).  Second, the equation considered by Bouin and Caillerie enjoys the maximum principle.  These two facts simplify many technical details in the long-range, long-time limit.  Also, the questions considered in the present work about the appropriateness of the reactive-telegraph equation, the difference in behavior with the non-linearity $\rho(1-\rho)$ versus $\rho(1-\rho)_+$, and the existence and dependence on dimension of the phase transition in $\tau$ are not considered in~\cite{BouinCaillerie}.  On the other hand, the work of Bouin and Caillerie considers a much broader class of kinetic models, which will be useful in studying a variety of biological models.
%Finally, during the write-up of this work, we became aware of a parallel work by Bouin and Caillerie~\cite{BouinCaillerie} which, at first glance, has  some  similarity  to the present one. In~\cite{BouinCaillerie}, the authors investigate the long-range, long-time limit of a related system with reaction term $\rho(M(v)-p)$ and with general dispersal kernel $M(v)$ and consider the question of existence and stability of traveling waves.  Their interest is to overcome technical difficulties associated with singular sets in the Hamiltonian associated to general $M(v)$.  As such, the assumptions in~\cite{BouinCaillerie}  are different than ours. Questions like those considered in this paper, namely, the appropriateness of the reactive-telegraph equation, other choices of non-linearities $F$, and the dependence on dimension of the phase transition in $\tau$ are not considered in~\cite{BouinCaillerie}. 
\smallskip

\subsection*{Organization of the paper} %\vskip-.125in 
The paper is organized into six main parts.  In \Cref{sec:proofs} we present  a preliminary lemma that gives an upper bound for $J^\e$ that is independent of $\e$ and prove the convergence part of \Cref{thm:propagation}. The Hamiltonian and its properties as well as the proof of \Cref{thm:phi_convergence} are given  in \Cref{sec:proofs1}.
%  we introduce the Hamiltonian, discuss some of its properties .  Theorem \Cref{thm:propagation} is proved in \Cref{sec:proofs1}.  
%In \Cref{sec:speeds} we  identify the Hamiltonian, analyze its properties and prove \Cref{thm:phi_convergence}.  Theorem \Cref{thm:propagation} is proved in \Cref{sec:proofs1}.  
In \Cref{sec:speeds} we prove \Cref{prop:speedn} and \Cref{prop:speed12}.  In \Cref{sec:tau} we show that we may take the hydrodynamic limit $\tau\to0$ to recover the Hamilton-Jacobi equation of the Fisher-KPP equation.
% that is we prove \Cref{prop:tau}.  
In \Cref{sec:positivity} we prove Theorem \ref{thm:non_positive_telegraph} and discuss an example of a related 
%we show that the positivity of the population densities is not preserved in the reactive-telegraph equation and a related 
kinetic system with logistic reaction term which does not preserve positivity either.  A brief Appendix is included containing computations of some non-standard integrals that are used throughout the manuscript.

\subsection*{The notion of solution} We are not making any assumptions for \eqref{takis1} and \eqref{eq:the_equation} to have smooth solutions. We interpret both equations as well as the limiting Hamilton-Jacobi variational inequalities and equations in the classical Crandall-Lions viscosity sense.

\subsubsection*{Acknowledgements}
We wish to thank the anonymous referees for a close reading of the manuscript and their very helpful comments.  CH would like to thank Jacek Jendrej for pointing out the refence~\cite{Sogge} regarding regularity of hyperbolic equations.
CH was partially supported by the National Science Foundation Research Training Group grant DMS-1246999.
PS was partially supported by the National Science Foundation grants DMS-1266383 and DMS-1600129 and the Office for Naval Research Grant N00014-17-1-2095.

%WE ONLY LOOK AT THE KINETIC SYSTEM ABOVE SINCE THE REACTIVE-TELEGRAPH EQUATION DOES NOT PRESERVE POSITIVITY.
%\begin{proposition}\label{prop:negative}
%	There exists $\rho_0 \in C^\infty(\R^3)$ with $0 \leq \rho_0 \leq 1$ and $t_0 \in \R_+$ and $x_0 \in \R$ such that the solution $\rho$ to~\eqref{eq:reactive_telegraph} with initial data $\rho(0,x) = \rho_0(x)$ satisfies $\rho(t_0,x_0) < 0$.  
%\end{proposition}
%TALK AOBUT OTHER DIMENSIONS

%ALSO THE KINETIC ONE HAS PROBLEMS.
%\begin{proposition}\label{prop:kinetic_negative}
%	There exists initial data $p_0$ and a solution $p$ to~\eqref{eq:the_equation} such that $0 \leq p_0, \rho_0 \leq 1$ on $\R^2$, where $\rho_0 = \rho(t,x)$, and such that there exists $t_0 >0$, $e\in S^{n-1}_d$, and $x_0 \in \R^2$ where $p(t_0,e,x_0) <  0$.
%\end{proposition}
%
%NOTICE THAT THIS SAYS THAT THERE IS A NEGATIVE LOCAL POPULATION DENSITY OF INDIVIDUALS MOVING IN DIRECTION $e$, WHICH IS NON-SENSICAL.

\section {Bounds and the proof of Theorem~\ref{thm:phi_convergence}} %, the Hamiltonian and its properties} %Theorem~\ref{thm:propagation}}
\label{sec:proofs}
\subsection*{The a priori bounds} 
%We drop here the super-script $\tau$ notation when discussing $p^\e$, $J^\e$, and $\rho^\e$ since we are not, for the moment, interested in the dependence on $\tau$. 
%\smallskip
%\textbf{TAKIS -- If we don't include the commented out paragraph here about dropping the tau notation, shouldn't we include the notation?}

%Before embarking on the proofs of \Cref{thm:phi_convergence} and \Cref{thm:propagation}, 
We state as a lemma the fact that \eqref{takis1} preserves positivity and yields an upper bound that is independent of $\e$. This is important, since without the  positivity of $p^\e$, $J^\e$ is not well-defined, while the upper bound is needed in order to study the limit $\e\to 0$.
\begin{lem}\label{takis10}\label{lem:a_priori}
Assume  $0\leq p_0\leq 1$ and let $p^\e$ be the solution to \eqref{takis1} and \eqref{takis3}. Then, for all $t\in \R_+$, 
$$0 \leq p^\e(\cdot, \cdot, t) \leq M_\tau, $$ where
\[ M_\tau = \begin{cases}
		1 \ \text{ if } \  \tau 
	\leq 1,\\[1.5mm]
		\frac{(1+\tau)^2}{4\tau} \   \ \text{ if } \  \tau >1.
	\end{cases}
 \]
%\[
%	M_\tau = \begin{cases}
%		1 \ \text{ if } \  \tau < 1,\\[1mm]
%		\frac{(1+\tau)^2}{4\tau} \ \text{ otherwise.}
%	\end{cases}
%\]
\end{lem}
%We remark that, if $\tau \leq 1$, then we may replace the upper bound $M_\tau$ with the upper bound $1$, which is more natural in view of the assumption on $p_0$. Since the proof is similar we omit this improvement.
\begin{proof}
The positivity of $p^\e$ follows from rewriting   \eqref{takis1} as
\[ p^\e_t + a_{n,\tau} v\cdot D p^\e + \frac{1}{\e \tau} p^\e= \frac{1}{\e \tau} \rho^\e +  \frac{1}{\e }\rho^\e(1-\rho^\e)_+,
\]
and observing that the right hand side of the above equation is always nonnegative.

%We first show the lower bound for $p^\e$. Re Set $\overline p^\e: = e^{-t/(\tau \e)} p^\e$ and observe that, for $\overline \rho^\e = \fint \overline p^\e dv$, 
%\[
%	\e\left(\overline p^\e_t + a_{n,\tau}v \cdot D \overline p^\e\right) = \frac{1}{\tau} \overline\rho^\e + F(e^{t/(\tau\e)}\overline\rho^\e),
%\]
%%where $\overline \rho^\e = \fint \overline p^\e dv$.  
%Since $F \geq 0$ and $\overline \rho^\e\geq 0$ if
%$\inf_v \overline p^\e\geq 0$,  it is clear that  $\inf_v \overline p^\e$ is increasing.  The claim now follows from the nonnegativity of $p_0.$ % Hence, $\overline p_0\geq 0$, which, in turn, implies that $p^\e$ is non-negative.
%\smallskip
\smallskip

For the upper bound, if $\tau >1$,  let $$q^\e:=\frac{(1+\tau)^2}{4\tau} -p^\e \ \ \text{and} \ \ \hat \rho^\e:= \fint_{S^{n-1}} q^\e(x,v,t)dv.$$ It follows that 
\[ q^\e_t + a_{n,\tau} v\cdot D q^\e + \frac{1}{\e\tau} q^\e= \frac{1}{\e \tau} \hat \rho^\e +  \frac{1}{\e }\left(\hat \rho^\e - \frac{(1+\tau)^2}{4\tau}\right)\left(\hat\rho^\e - \left(\frac{(1+\tau)^2}{4\tau} -1\right)\right)_+.
\]
Straightforward calculations  yield that the right hand side above is nonnegative, which, in turn, implies that $q^\e \geq 0$ and, hence, the claim.
\smallskip

When $\tau \leq 1$ the proof is similar, so we omit the details. 
\end{proof}

\subsection*{Propagation}

We present  the proof of  Theorem~\ref{thm:phi_convergence}, which is an immediate consequence of \Cref{thm:propagation}. 
%\smallskip

%Since we are not, for the moment, interested in the dependence on $\tau$,  here we drop the super-script $\tau$ notation when discussing $J$.
%% since we are not, for the moment, interested in the dependence on $\tau$. 

\begin{proof}[Proof of \Cref{thm:phi_convergence}]
Fix any $(x_0,t_0) \in \{J<0\}$ with $t_0 > 0$.  Since $J^\e$ converges locally uniformly in $x,t$ and uniformly in $v$ to $J$, then, for all $\e$ sufficiently small, we have 
		\[ 
		\sup_{v \in S^{n-1}}J^\e( x_0,v, t_0) \leq J(x_0,t_0)/2 < 0.
	\]
It follows that 
	\[
	\rho^\e(x_0,t_0)= \fint_{S^{n-1}} p^\e(x_0,v,t_0)dv\leq 	\sup_{v \in S^{n-1}}p^\e(x_0,v, t_0) =e^{\frac{1}{\e}\sup_{v \in S^{n-1}}J^\e(v,x_0,t_0)}
			\leq e^{\frac{J(x_0,t_0)}{2\e}},
	\]
and, hence, $\lim_{\e\to 0} \rho^\e(x_0,t_0) =0$; the local uniformity of the limit is immediate.	
\smallskip
	
	Now consider any point $(x_0,t_0) \in \Int \{J = 0\}$ with $t_0 > 0$.  Let $\phi(x,t) = -|t - t_0|^2 - |x-x_0|^2$ and notice that $J - \phi$ has a strict local minimum at $(x_0,t_0)$.  Hence, there exist  $( x_\e,t_\e)$ converging to $(x_0, t_0)$ such that $\min_v J^\e(x,t) - \psi$ has a local minimum at $(x_\e,t_\e)$.  Let $v_\e$ be such that $J^\e(x_\e,v_\e,t_\e) = \min_ v J^\e(x_\e,v,t_\e)$.
	%
	 %for some $v_0 \in S^{n-1}$ such that $J^\e - \psi$ has a local in $(x,t)$ and global in $v$  minimum at $( x_\e, v_\e, t_\e)$.
\smallskip
	  
	  Uaing~\eqref{eq:the_equation},  we find %, at $(x_\e,v_\e, t_\e)$,
	\begin{equation}\label{eq:rho_inequality}
	%\begin{split}
		\phi_t (x_\e,t_\e)+ a_{n,\tau} v_\e \cdot D \phi (x_\e,t_\e) \geq \frac{1}{\tau} \left(\frac {\rho^\e(x_\e,t_\e)}{p^\e(x_\e,v_\e,t_\e)} -1\right) + \frac {\rho^\e(x_\e,t_\e)}{p^\e(x_e,v_\e,t_\e)}	(1 - \rho^\e(x_\e,t_\e))_+. \\[1.5mm]
%		&\frac{1}{\tau} \left(\frac {\rho^\e(x_\e,t_\e)}{p^\e(x_\e,v_\e,t_\e)} -1\right) + \frac {\rho^\e(x_\e,t_\e)}{p^\e(x_e,v_\e,t_\e)}	(1 - \rho^\e(x_\e,t_\e))_+.
%			
%			& \frac{1}{\tau} \left( \fint_{S^{n-1}} e^{\frac{J^\e( x_\e,v, t_\e) - J^\e(x_\e,v_\e, t_\e)}{\e}} dv - 1\right) + \fint_{S^{n-1}} e^{\frac{J^\e(x_\e, v, t_\e) - J^\e(x_\e,v_\e, t_\e)}{\e}} dv (1 - \rho^\e)_+.
	%\end{split}
	\end{equation}
Observe that, since $(x_\e,v_\e,t_\e)$ is the location of a global minimum in $v$ of $J^\e(x_\e, v_\e, t_\e)$, 
$$\frac {\rho^\e(x_\e,t_\e)}{p^\e(x_e,v_\e,t_\e)} \geq 1.$$
Moreover, 
%we findf and that $\psi$ does not depend on $v$, we find, at $(x_\e,t_\e)$ 
%	\[\begin{split}
%		\psi_t + a_{n,\tau} v_\e \cdot D \psi
%			&\geq (1 - \rho^\e)_+,
%	\end{split}\]
an explicit computation implies that, as $\e\to 0$,   $\phi_t (x_\e,t_\e)+ a_{n,\tau} v_\e \cdot D \psi (x_\e,t_\e)  \to 0.$ %\eqref{eq:rho_inequality} tends to zero as $\e$ tends to zero. 
\smallskip

 Then 
 \begin{equation}\label{eq:rho_p_limit}
 	 \lim_{\e\to 0} \frac {\rho^\e(x_\e,t_\e)}{p^\e(x_\e,v_\e,t_\e)}=1 \ \ \text{and} \  \ \lim_{\e\to 0}(1 - \rho^\e(x_\e,t_\e))_+=0.
\end{equation}
The second limit above implies that   $\liminf_{\e \to 0}\rho^\e(x_\e,t_\e) \geq 1$. %and, using this information in the other limit, we get that $\limsup_{\e\to 0}  p^\e(x_e,v_\e,t_\e) \geq 1.$
\smallskip

%Returning to~\eqref{eq:rho_inequality}, by the previous argument, the second term on the right-hand side tends to zero so the first term on the right hand side must tend to a non-positive number.  We, thus, obtain
%\[
%	\frac{\rho^\e(x_\e,t_\e)}{p^\e(x_\e,v_\e,t_\e)}
%		= \fint_{S^{n-1}} e^{\frac{J^\e( x_\e,v, t_\e) - J^\e(x_\e,v_\e, t_\e)}{\e}} dv
%		\leq 1,
%\]
%which, in turn, yields $\liminf_{\e\to0}p^\e(x_\e,v_\e,t_\e) \geq 1$.
%\smallskip
%
Next, recall that, for all $v\in S^{n-1}$, $J^\e(x_0,v,t_0) \geq J^\e(x_\e,v_\e,t_\e) - \phi(x_\e,t_\e)$.  Hence  %Using this, we have
\[\begin{split}
	\rho^\e(x_0,t_0)
		&= \fint_{S^{n-1}} e^{\frac{J^\e(x_0,v,t_0)}{\e}}dv
		\geq \fint_{S^{n-1}} e^{\frac{J^\e(x_\e,v_\e,t_\e) + |x_\e - x_0|^2 + |t_\e-t_0|^2}{\e}}dv\\[2mm]
		&\geq p^\e(x_\e,v_\e,t_\e)
		= \left(\frac{p^\e(x_\e,v_\e,t_\e)}{\rho^\e(x_\e,t_\e)}\right) \rho^\e(x_\e,t_\e).
\end{split}
\]
Letting  $\e\to0$ and using~\eqref{eq:rho_p_limit}, %that $\limsup_{\e\to0} p^\e(x_\e,v_\e,t_\e) \geq 1$,
we conclude that%, as claimed when $\tau >1,$
\beq\label{takis195}
	\liminf_{\e\to0} \rho^\e(x_0,t_0) \geq 1.
\eeq
%as claimed when $\tau >1.$
%\smallskip

Since, when $\tau \leq 1$,  $p^\e \leq 1$, we have $\rho^\e \leq 1$, \eqref{takis195} yields  $$\lim_{\e\to0} \rho^\e(x_0,t_0)=1.$$

%
%
%
%  It follows that  {\bf I think this only shows that the limsup is greater than or equal to 1}
%\begin{equation}\label{takis21}\liminf_{\e\to 0}\rho^\e(x_\e,t_\e) \geq 1.\end{equation}
%%\smallskip
%It $\tau\leq 1$, we know that $p^\e\leq 1$ and, hence, $\rho^\e\leq 1$. Combining this with \eqref{takis21} yields the claim.
%%
%%On the other hand, by our choice of $(x_\e,v_\e,t_\e)$, we have that, for any $v\in S^{n-1}$,
%%\begin{equation}\label{eq:liminf_to_1}
%%	J^\e(x_\e,v_\e,t_\e) - \psi(x_\e,t_\e)
%%		\leq \min_{v' }J(x_0,v',t_0) - \psi(x_0,t_0)
%%		\leq J(x_0,v,t_0),
%%\end{equation}
%%since $(x_\e, t_\e)$ is the location of a local minimum of $\min_{v'} J^\e$ and $\psi(x_0,t_0) = 0$.  Hence,
%%\begin{equation}\label{eq:liminf_compare}
%%	\rho^\e(x_0,t_0)
%%		\geq \rho^\e(x_\e,t_\e) \exp\left\{ \frac{1}{\e}(|x_\e - x_0|^2 + |t_\e -t_0|^2)\right\}.
%%\end{equation}
%%NEED TO CHANGE THIS BECAUSE OF THE NEW UPPER BOUND.  An easy application of the maximum principle implies that $\rho^\e \leq 1$.  The combination of this, \eqref{eq:liminf_to_1}, \eqref{eq:liminf_compare}, yields
%%\[
%%	\lim_{\e\to0} \rho^\e(x_0,t_0) = 1.
%%\]
\end{proof}

\section{The Hamiltonian, its properties and the proof of Theorem~\ref{thm:propagation}}
\label{sec:proofs1}

\subsection*{The Hamiltonian $H$ and its properties}\label{sec:hamiltonian}

To motivate the choice of the Hamiltonian $H$, we first present a formal argument about the limit of the $J^\e$ % in $\{J<0\}$
%in \eqref {eq:the_equation}, 
assuming that 
$$J^\e(x,v,t)= J(x,t) + \e \eta(x,v,t) + \text{o}(\e).$$
Working in $\{ J<0 \}$, where we can ignore  $\rho^\e$ in $(1-\rho^\e)_+$, we get from  \eqref {eq:the_equation} that
$$J_t + a_{n,\tau} v\cdot DJ = \frac{1}{\tau} \left( \fint_{S^{n-1}} e^{\eta(x,v',t) - \eta(x,v,t) }dv' - 1\right) %\\[4.5mm]  \qquad & 
+  \fint_{S^{n-1}} e^{\eta(x,v',t) - \eta(x,v,t)} dv'.   $$% (1-\rho^\e)_+
Since $ J_t $ is independent of $v$, there must exist some, independent of $v$, constant $H (  DJ)$ so that, for all $v\in S^{n-1}$, 
$$- H (  DJ) +\frac{1}{\tau}  +  a_{n,\tau}  v\cdot DJ  = \frac{{\tau} + 1}{\tau} \fint_{S^{n-1}} e^{\eta(x,v',t) - \eta(x,v,t) }dv'. 
$$  %\\[4.5mm]  \qquad & 
%+  \fint e^{\psi(x,v',t) - \psi (x,v,t))} dv = 
The above expresssion leads to the cell (eigenvalue)  problem to find, for each $p\in \R^n$, a unique constant $H(p)$ (eigenvalue) and some $\eta=\eta(v;p)$ (eigenfunction) such that, for all $v\in S^{n-1}$, 
\begin{equation}\label{takis100}%(-H(p) +\frac{1}{\tau} + a_{n,\tau} v\cdot p) 
\frac{e^{\eta(v)}}{\fint_{S^{n-1}} e^{\eta(v') }dv'}=  \frac{\tau +1}{\tau} \frac{1}{-H(p) +\frac{1}{\tau} + a_{n,\tau} v\cdot p}.%
\end{equation}

It follows that, if it exists, $H(p)$ must be defined implicitly by  
\begin{equation}\label{eq:integral_form}
	\frac{\tau}{1+\tau} = \fint_{S^{n-1}} \frac{dv}{-H(p) +\frac{1}{\tau}  +  a_{n,\tau} v\cdot p}.
\end{equation}

%$a_{n,\tau}$
%Next we show how it is possible to find such $H$ and $\eta$.  We begin with the function 
Consider the function $\Phi: [1,\infty) \to \overline \R_+$  given by 
\begin{equation}\label{eq:Phi}
	\Phi(s) := \fint_{S^{n-1}} \frac{dv}{s + v_1};
\end{equation}
here $v = (v_1,v_2,\dots,v_n)$.  It is immediate that $\Phi' <0$ and $\lim_{s \to\infty} \Phi(s)=0$. 
%and observe that it is strictly decreasing and tends to zero as $s \to\infty$;  
Moreover, as it is shown in the Appendix, 
%\smallskip
 $\Phi(1) = \infty$ when $n \in \{1,2,3\}$, while $\Phi(1) < \infty$ for $n > 3$.
\smallskip

Fix $p\in \R^n\setminus \{0\}$.  Then, looking at \eqref{eq:integral_form} and the properties of $\Phi$ we assert that 
\begin{equation}\label{takis30}
H(p):=\begin{cases} -  a_{n,\tau}|p| +\frac{1}{\tau} \ \ \  &\text{ if} \ \ \Phi(1) \leq  \frac{\tau }{1+\tau} a_{n,\tau}|p|,\\[3mm]
\alpha \ \ &\text{ if} \  \ \Phi(1) \geq  \frac{\tau }{1+\tau} a_{n,\tau}|p|,
\end{cases}
\end{equation}
where, in the latter case, $\alpha$ is the unique 
%
% If
%\[
%	A(1) \geq \frac{\sqrt{n\tau}}{1 + \tau} |p|,
%\]
%$H(p)$ to is the unique 
negative number such that
\begin{equation}\label{takis102}
	\Phi\left( \frac{ -\alpha +\frac{1}{\tau} }{a_{n,\tau}|p|}\right) = \frac{\tau }{1+\tau} a_{n,\tau}|p|.
\end{equation}
%note that this formula  is equivalent to \eqref{eq:integral_form}. % and, in addition, that H
%\smallskip
Note  that $H$ is continuous, isotropic, that is depends only on $|p|$,  and 
%\smallskip
%and %Letting $p\to 0$ in \eqref{eq:integral_form}, gives
$H (0)= -1.$ %    - \frac{1+\tau}{\tau}.$$
We note that this Hamiltonian is the same as found in~\cite{BouinCaillerie,Caillerie}.

%Finally notice that $H$ is isotropic, that is depends only on $|p|$. 
\smallskip
%otherwise 
%%$$H(p) = - a_{n,\tau} |p|$$
%
%\begin{equation}\label{eq:integral_form}
%	\frac{\tau}{1+\tau} = \fint_{S^{n-1}} \frac{dv}{-H(p) + a_{n,\tau} v\cdot p}.
%\end{equation}
%Otherwise, define $H(p) = - a_{n,\tau} |p|$.
%\smallskip

%Next we verify that the above defined $H(p)$ is the sought after constant.  For this we need to find $\eta$ such that \eqref{takis100} is satisfied.
%\smallskip
%
%

Next we show that \eqref{takis30} is  indeed  correct  for every $p$, when $n \in \{1,2,3\}$, and for all $p$ such that  $\Phi(1) >  \frac{\tau }{1+\tau} a_{n,\tau}|p|$, if $n>3.$ We present the argument in the latter case, since the discussion applies to the former.

\smallskip 
If   $\Phi(1) >  \frac{\tau }{1+\tau} a_{n,\tau}|p|$, then $\eta: S^{n-1} \to \R$ given by
\begin{equation}\label{takis101}
e^{\eta(u)}=\frac{\tau+1}{\tau} \frac{1}{-H(p) +\frac{1}{\tau} +a_{n,\tau} v\cdot p},
\end{equation} 
with $H(p)$ given by the second alternative in \eqref{takis30},  clearly satisfies \eqref{takis100}. 
%\smallskip
If   $\Phi(1) =  \frac{\tau }{1+\tau} a_{n,\tau}|p|$, then the $\eta$ given in \eqref{takis101} also satisfies \eqref{takis100} for all
$v\in S^{n-1}$ except  $v=-p/|p|$.
\smallskip
%\vskip.075in

It is clear that, when $\Phi(1) <  \frac{\tau }{1+\tau} a_{n,\tau}|p|$,   it is not possible to find such an $\eta$. 
% if  $\Phi(1) \leq  \frac{\tau }{1+\tau} a_{n,\tau}|p|$.
%\smallskip
In the proofs, we deal with this issue by considering ``approximate'' correctors, which for $\delta >0$ and appropriately chosen $\mu >0$, are given by 
\beq\label{190}
\eta_{\mu,\delta}(v):=\frac{\tau +1}{\tau} \frac{1}{{(\delta-H(p) +\frac{1}{\tau} +a_{n,\tau} v\cdot p})^\mu}.
\eeq
%\qed

%It follows from \eqref{takis102} that
%\[
%\fint_{S^{n-1}} e^{\eta(x,v',t)}dv' =1
%\]
%and, hence, \eqref{takis101} is satisfied. 
%\smallskip
%
%The case $\Phi(1) \leq  \frac{\tau }{1+\tau} a_{n,\tau}|p|$ is  more complicated. Here we need to find $\eta$ 
%such that 
%\begin{equation}\label{takis103}
%e^{\eta(v)} (a_{n,\tau} |p|  + a_{n,\tau} v\cdot p)=\frac{\tau}{\tau + 1}\fint e^{\eta(x,v',t)}dv'.
%\end{equation}
%Reorganizing \eqref{takis103} and letting  $\hat p=\frac{p}{|p|}$ leads to 
%\begin{equation}\label{takis104}
%\frac{e^{\eta(v)}}{\fint e^{\eta(x,v',t)}dv'}=\frac{\tau}{(\tau + 1) a_{n,\tau}|p|} \frac{1}{1 + v\cdot \hat p};
%\end{equation} 
%notice that the last fraction in the right hand side of \eqref{takis104} is well defined except when $v=-\hat p$.
%% but away from this direction $\eta$ is well defined.
%\smallskip
%
%It is now immediate that should such $\eta$ exist, we would require 
%$$
%a_{n,\tau} \frac{\tau}{\tau +1} |p|= \Phi(1),
%$$
%which is not possible. Nevertheless, as we show in the next subsection, the choice of $H(p)$ in this case is the correct one.
\smallskip
We now discuss the concavity of $H$.  We present the argument for $n>3$, since the other cases follow similarly without having to deal with the first part of the definition of the Hamiltonian.
\smallskip

When $|p| >  (1+\tau)\Phi(1)/(a_{n,\tau}\tau)$,  the concavity is obvious from \eqref{takis30}. 
%since $H(p) = -a_{n,\tau}|p|$.  
%\smallskip
%\vskip.075in
When $|p| < (1+\tau)\Phi(1)/ (a_{n,\tau} \tau)$,
%argue as follows;  in the computations, we assume that $H$ is smooth, but, if not, one may argue similarly using difference quotients.  
%\smallskip
we use~\eqref{eq:integral_form} to obtain
\begin{equation}\label{eq:DpH}
	\fint_{S^{n-1}} \frac{D_p H(p) - a_{n,\tau} v}{\left(- H(p) +\frac{1}{\tau}+a_{n,\tau}v \cdot p\right)^2}dv = 0.
\end{equation}
Differentiating again we get
\[
	\fint_{S^{n-1}} \frac{ D_p^2 H(p) dv}{\left(-H(p) +\frac{1}{\tau} + a_{n,\tau} v\cdot p\right)^2}
		= - 2\fint_{S^{n-1}}\frac{\left(D_p H(p)  - a_{n,\tau} v \right)\otimes\left(D_p H(p)  -  a_{n,\tau} v \right)}{\left( - H(p) +\frac{1}{\tau}+  a_{n,\tau}v \cdot p\right)^3}dv.
\]
Fix any non-zero vector $\xi \in \R^n$. Then
\[
	 D_p^2 H(p)\xi \cdot \xi  \fint_{S^{n-1}} \frac{ dv}{\left(-H(p) +\frac{1}{\tau} + a_{n,\tau} v\cdot p\right)^2}\\
		= - 2\fint_{S^{n-1}}\frac{\left|\left(D_p H(p)  - a_{n,\tau} v \right) \cdot \xi\right|^2}{\left( - H(p) +\frac{1}{\tau} + a_{n,\tau} v \cdot p\right)^3}dv, 
\]
and  concavity follows after noticing that, by its construction, $-H(p)+\frac{1}{\tau} - a_{n,\tau} |p|\geq0$.
 \smallskip
 
Finally to conclude we need to show that the concavity property is preserved across  $|p| = \frac{1+\tau}{\tau a_{n,\tau}}\Phi(1)$ when $\Phi(1) < \infty$. Since $H$ is isotropic, this is immediate if we show 
that 
%it i radial and decreasing.  Hence, to show concavity at the interface, we simply need to show that
\begin{equation}\label{chris21}
	\limsup_{|p| \to \frac{1+\tau}{\tau a_{n,\tau}}\Phi(1)^-} |D_p H(p)|
	%(1+\tau)\Phi(1)/(a_{n,\tau}\tau)} |D_p H(p)|
		\leq \lim_{|p| \to  \frac{1+\tau}{\tau a_{n,\tau}}\Phi(1)^+} |D_p H(p)| = a_{n,\tau}.
\end{equation}
%We now establish~\eqref{chris21}.
%\smallskip

To establish ~\eqref{chris21} we consider two cases depending on whether  $\fint_{S^{n-1}} (1+v\cdot \hat p)^{-2} dv$ is finite or not. %, where $\hat p = p/|p|$.  
Note that this quantity does not depend on $\hat p$, since we may simply change variables in the integral. 
\smallskip
%$$\hat p = p/|p|$$. 

If $\fint_{S^{n-1}} (1+v\cdot\hat p)^{-2} dv < \infty$ then, in view of~\eqref{eq:DpH}, we get %$D_p H(p)$ is given by
\[
	\limsup_{|p| \to \frac{1+\tau}{\tau a_{n,\tau}}\Phi(1)^-} |D_p H(p)|
		= \left| \fint_{S^{n-1}} \frac{a_{n,\tau} v}{\left(1+v \cdot \hat p\right)^2}dv\right|
			\left( \fint_{S^{n-1}} \frac{1}{\left(1+v \cdot \hat p\right)^2}dv\right)^{-1}\leq a_{n,\tau}.
\]
%Taking the absolute value of both sides, we obtain
%\[
%	\lim_{|p| \nearrow (1+\tau)\Phi(1)/(a_{n,\tau}\tau)}|D_p H(p)|
%		\leq a_{n,\tau},
%\]
%which gives exactly \eqref{chris21}.
%\smallskip

If $\fint_{S^{n-1}}(1+v\cdot \hat p)^{-2}dv = \infty$, then, using~\eqref{eq:DpH}, we find 
\[
	\lim_{|p| \to \frac{1+\tau}{\tau a_{n,\tau}}\Phi(1)^-} |D_p H(p)| = a_{n,\tau}, 
\]
since, otherwise, the left hand side of~\eqref{eq:DpH} is not finite, and the claim follows. 
%\smallskip
%
%It then follows that 
%\[
%	\lim_{|p| \to \frac{1+\tau}{\tau a_{n,\tau}}\Phi(1)^-}|D_p H(p)|
%		= a_{n,\tau},
%\]
%and the proof is complete.
{%\subsection*{An upper bound}
\subsection*{The half-relaxed limits and the limiting equation}\label{sec:half_relaxed}
We prove \Cref{thm:phi_convergence}.  The  main tools  are the classical  (in the theory of viscosity solutions) half-relaxed limits~\cite{BarlesPerthame} and the methodology of \cite{BarlesEvansSouganidis}. The problem is that, since  $J^\epsilon$ takes infinite values,  we do not have uniform bounds.  To circumvent this we use an argument introduced  in~\cite{BarlesSouganidis} to deal with this kind of difficulty. % fact that $J^\epsilon$ takes infinite values.
\smallskip

We begin with the definition of the half-relaxed limits. Given a family $f_\e$ of uniformly bounded functions depending only on $x$ and $t$, the half-relaxed upper and lower limits $\overline f$  and $\underline f$ are respectively 
$$\overline f(x,t):= \limsup f_\e(x,t):=\limsup_{y\to x, s\to t, \e\to 0} f_\e(y,s) \ \ \text{and} \ \ \underline f(x,t):= \liminf  f_\e(x,t):=\liminf_{y\to x, s\to t,\e\to 0} f_\e(y,s).$$
We also remark that,  in view of Lemma~\ref{takis10}, %there exists $C_\tau\qeq 0$ such that 
\begin{equation}\label{takis140}
J^\e \leq \e \log M_\tau  \ \ \text{in}  \ \  \R^n\times S^{n-1}\times \R_+.
\end{equation}
%when $\tau \leq 1$, $C_\tau=0$, while, when $\tau>1$, $C

Fix $A>0$ and let  $$J^\e_A: = \max\{J^\epsilon,-A\}.$$ 
%Recall further that $J^\e_A$ is a viscosity solution to
%\begin{equation}\label{eq:phi_epsilon_A}
%	\phi_t + \frac{1}{\sqrt \tau} v \cdot D \phi
%		= \frac{1}{\tau}\left(\fint_{S^{n-1}} e^{\frac{1}{\epsilon}\left(\phi(v',x,t) - \phi(v,x,t) \right)} dv' - 1\right)
%			+ r\left(\fint_{S^{n-1}} e^{\frac{1}{\epsilon}\left(\phi(v',x,t) - \phi(v,x,t) \right)} dv'\right)(1 - \rho)_+.
%\end{equation}
Since $J^\e_A$ is bounded, uniformly in $\e$,  from above and below by $\e\log M_\tau$ and $-A$ respectively, 
we  take the half-relaxed upper and lower limits  of $\max_v J^\e_A$ and $\min_v J^\e_A$ respectively; that is, we consider $\overline J_A$ and $\underline J_A$ given by 
\begin{equation}\label{eq:half_relaxed_A}
	\overline J_A(x,t): = \limsup \max_v J^\e_A(x, v,t) \ \text{ and }\ \underline J_A(x,t): = \liminf \min_v J^\e_A(x,v,t).
%		\qquad \text{ and } \qquad
%	\overline J_A(x,t): = \limsup \max_v J^\e_A(v,y,s).	
\end{equation}
%note that, by construction, 
%\beq\label{takis170}
%\underline J_A \leq \overline J_A.
%\eeq
We point out that $\max_v J_A^\e$ and $\min_v J_A^\e$ are uniformly bounded in $\e$ and do not depend on $v$.  Hence, their half-relaxed limits are well-defined.  Further, we note that we expect the limit of $J^\e_A$ to be independent of $v$ as fluctuations in $v$ will ``average out.'' This suggests that we lose no information in taking the maximum and minimum in $v$.
\smallskip

To state the next lemma we introduce some additional notation.  First, $J_{A,t}$ refers to the time derivative of $J_A$, and this notation applies similarly to other terms derived from $J_A$.  Second, given $g:\R^n\to \R$, $g^\star$ and $g_\star$ are respectively its upper and lower semicontinuous envelopes. Moreover, %the following notation
\[
{\bf 1}_A:=\begin{cases} 0 \ \ \text{on} \ \ G_0, \\[1.5mm]
-A \ \ \text{on} \ \ (\R^n \setminus \overline G_0).
\end{cases}
\]
\begin{lemma}\label{lem:sub_super}
Assume \eqref{assumption:G_0}. Then:

(i)~$\overline J_A$ is a (viscosity) sub-solution to 
\begin{equation}\label{eq:J_A2}
\begin{cases}	\max\left\{\overline J_{A,t} + H(D \overline J_A),
		\overline J_A,  - \overline J_A - A\right\} \leq 0 \ \  \text{in} \ \   \R^n\times \R_+,\\[1.5mm]
\overline J_A(\cdot,0)\leq ({{\bf 1}_A})^\star \ \text{on} \  \R^n;
\end{cases}
\end{equation}
(ii)~$\underline J_A$ is a (viscosity) super-solution to 
\begin{equation}\label{eq:J_A}
\begin{cases}	\max\left\{\underline J_{A,t} + H(D \underline J_A),
		\underline J_A, - \underline J_A  -A\right\} \geq 0 \ \  \text{in} \ \   \R^n\times \R_+,\\[1.5mm]
\underline J_A(\cdot,0)\geq  ({{\bf 1}_A})_\star \ \text{on} \  \R^n.
\end{cases}
\end{equation}

%$\underline J_A$ and $\overline J_A$ defined by~\eqref{eq:half_relaxed_A} are, respectively, super- and sub-solution to 
%\begin{equation}\label{eq:J_A}
%\begin{cases}	\max\left\{(J_{A,t} + H(DJ_A^\tau)+ \frac{1}{\tau},%\frac{\tau}{1+\tau} - \fint_{S^{n-1}}\frac{dv}{\frac{1}{\tau} + \phi_t + a_{n,\tau}v\cdot D \phi},
%		J_A^\tau, -A - J_A^\tau\right\} = 0 \ \  \text{in} \ \   \R^n\times S^{n-1}\times \R_+,\\[1.5mm]
%J_A^\tau=
%
%\end{equation}
%with the initial conditions $J_A^\tau(t=0,x) = -A \1_{G_0}$, in the viscosity sense.
\end{lemma}

The proofs of \eqref{eq:J_A2} and \eqref{eq:J_A} are similar. We separate into cases based on how the Hamiltonian is defined.  When $H$ is defined by the second case of~\eqref{takis30}, the proof is based on perturbing the test function by a small multiple of the corrector.  This method is classical, dating back to the work of Evans~\cite{Evans} (see also~\cite{MajdaSouganidis} for another early use and \cite{BouinCaillerie,BouinCalvezEikonal} for an application in a very similar context).  On the other hand, when $H$ is defined by the first case of~\eqref{takis30}, additional care is required because, as per the discussion in the previous subsection, exact correctors do not exist and a few more additional arguments are need in that case.  %Additional care is required when, as per the discussion in the previous subsection, correctors do not exist and a few more additional arguments  are needed in that case.
This is where the proofs of~\eqref{eq:J_A2} and~\eqref{eq:J_A} differ.
\smallskip

The proofs of parts (i) and (ii) are quite long and involved, and for the readers convenience we separate them into two proofs.  We first show the proof of (i), which is slightly simpler.

\begin{proof}[Proof of \Cref{lem:sub_super}(i)]\ \smallskip \\
\textbf{\# The proof of~\eqref{eq:J_A2} for positive times:}
%We first show the proof of  \eqref{eq:J_A2}, which is slightly simpler. % The proof of  \eqref{eq:J_A2} is similar,  so we  omit it.
%and for this it is convenient to have an inequality satisfied by $J^\e_A$.
\smallskip

Since $\rho^\e \geq 0$,  it follows from  \eqref{eq:the_equation} that
$$J^\e_{t} + a_{n,\tau}  v\cdot DJ^\e \leq \frac{1}{\tau} \left (\fint_{S^{n-1}} e^{\frac{1}{\e}( J^\e(\cdot,v',\cdot) - J^\e(\cdot,v, \cdot))} dv'-1 \right) +  \fint_{S^{n-1}} e^{\frac{1}{\e}( J^\e(\cdot,v',\cdot)-J^\e(\cdot,v, \cdot))} dv'.$$ 
Noting that  $-A$ also solves the inequality above and using that the maximum of two sub-solutions is itself a sub-solution  implies that 
\begin{equation}\label{takis150}
J_{A,t}^\e + a_{n,\tau}  v\cdot DJ^\e_A \leq \frac{1}{\tau} \left (\fint_{S^{n-1}} e^{\frac{1}{\e}( J^\e_A(\cdot,v',\cdot) - J^\e_A(\cdot,v, \cdot))} dv'-1 \right) +  \fint_{S^{n-1}} e^{\frac{1}{\e}( J^\e_A(\cdot,v',\cdot)-J^\e_A(\cdot,v, \cdot))} dv'.
\end{equation}

Let $\phi$ be a smooth test function and assume that $\overline J_A - \phi$ has a strict maximum at $(x_0,t_0)$
and  $\phi(x_0,t_0) = \overline J_A(x_0,t_0).$
\smallskip

We first assume that $t_0>0$ and claim that 
\begin{equation}\label{takis106}
	\max\left\{ \phi_t(x_0,t_0) + H(D\phi(x_0,t_0) ),%\frac{\tau}{1+\tau} - \fint_{S^{n-1}}\frac{dv}{\frac{1}{\tau} + \phi_t + a_{n,\tau}v\cdot D \phi},
		\overline J_A(x_0,t_0),   -\overline J_A(x_0,t_0) -A \right\} \leq 0.
\end{equation}

That $\max \left\{\overline J_A(x_0,t_0), - \overline J_A(x_0,t_0) -A\right\}\leq 0$ is an immediate consequence of the definition of $\overline J_A$ and \eqref{takis140}.
\smallskip
 
 Next we show that 
 \begin{equation} \label{takis141}
 \phi_t(x_0,t_0) + H(D\phi(x_0,t_0) )\leq 0.
 \end{equation}
 
% Since \eqref{takis106} is immediate if  $\underline J_A(x_0,t_0)$ is either $0$ or $-A$, we 
%assume that $\underline J_A(x_0,t_0) \in (-A,0)$, and seek to show that 
%%\smallskip
%\begin{equation}\label{takis130}
%\phi_t(x_0,t_0) + H(D\phi(x_0,t_0) )+ \frac{1}{\tau} \leq 0.
%\end{equation}
%
To simplify the presentation and shorten some formulae in what follows we write
$$p_0:=D\phi(x_0,t_0)  \ \ \text{and} \ \ \hat p_0 :=\frac{D\phi(x_0,t_0)}{|D\phi(x_0,t_0)|}.$$
There are two cases.
\smallskip

%In view of the definition of $H$ it is clear that in what follows we need to consider two cases. 
\textbf{Case one:} If  
$$ \Phi(1) >  \frac{\tau }{1+\tau} a_{n,\tau}|p_0|,$$
we consider the perturbed test $\phi^\e=\phi + \e \eta(v; p_0)$ with $\eta$ given by \eqref{takis101}.
%Define
%\begin{equation}\label{eq:micro_corrector}
%	\eta_\e(x,v,t) = -\log\left(\mu_\e(x,t) + a_{n,\tau} \left(\frac{D \psi}{|D \psi|} + v\right) \cdot D \psi(x,t)\right),
%\end{equation}
%for $\mu_\e$ to be determined such that $\e\eta_\e$ is uniformly bounded and $\e D\eta_\e$ and $\e \partial_t \eta$ 
\smallskip

Let  $(x_\e,t_\e)$ be a maximum point  of $\max_vJ^\e_A - \phi^\epsilon$ in a neighborhood of $(x_0,t_0)$.  The use of such an approximating sequence in the theory of viscosity solutions is standard (see, for example,~\cite{Barles,LionsBook}).  In short, using the boundedness and continuity of $\eta$ along with the fact that $(x_0,t_0)$ is a strict local maximum of $\overline J_A - \phi$, its existence follows directly from the definition of $\limsup$.
\smallskip

Then there exist $v_\e \in S^{n-1}$ such that $(x_\e,v_\e,t_\e)$ is a maximum point of $J^\e_A -\phi^\e$ and  along a subsequence, which we denote the same way,  $\e\to 0$, there exists  $v_0\in S^{n-1}$ such that $(x_\e,v_\e, t_\e) \to (x_0, v_0, t_0)$ and $J^\e_A (x_\e,v_\e, t_\e) \to \overline J_A (x_0, t_0).$  We note that the existence of the limiting vector $v_0 \in S^{n-1}$ follows from the compactness of $S^{n-1}$.

\smallskip

It follows that, for $\e$ small enough and all $v \in S^{n-1}$ and all $(x,t)$ in a small ball $B_r(x_0,t_0)$ of radius $r >0$, 
\begin{equation}\label{eq:test_function_difference}
	\phi^\e(x,v,t) - \phi^\e(x_\e,v_\e, t_\e)
		\geq J^\e_A(x,v,t) - J^\e_A(x_\e,v_\e,t_\e).
		\end{equation}

Then \eqref{takis150} implies 
\begin{equation}\label{takis108}
%\begin{split}
	\phi_t^\e(x_\e,v_\e,t_\e) + a_{n,\tau} v_\e \cdot D\phi^\e (x_\e,v_\e,t_\e) +\frac{1}{\tau}
		%\leq  J^\e_t + a_{n,\tau} v \cdot D J^\e\\
%		&\geq \frac{1}{\tau}\left( \fint_{S^{n-1}} e^{\frac{1}{\e}( J^{\e}(x_\e,v',t_\e) - J^\e(x_\e,v_\e,t_\e))} dv' - 1\right)
%			+ \fint_{S^{n-1}} e^{\frac{1}{\e}( J^\e(x_\e,v',t_\e) - J^\e(x_\e,v_\e,t_\e))} dv'.
\leq \frac{\tau +1}{\tau}\fint_{S^{n-1}} e^{\eta(v')-\eta(v_\e)} dv'.
% - 1\right)
%			+ \fint_{S^{n-1}} e^{\frac{1}{\e}( J^\e(x_\e,v',t_\e) - J^\e(x_\e,v_\e,t_\e))} dv',
%			\\
%		&\geq \frac{1}{\tau}\left( \fint_{S^{n-1}} e^{\eta(x_\e,v',t_\e) - \eta(x_\e,v_\e,t_\e)} dv' - 1\right)
%			+ \fint_{S^{n-1}} e^{ \eta(x_\e,v',t_\e) - \eta(x_\e,v_\e,t_\e))} dv'(1 - \rho)_+.
%\end{split}
\end{equation}
%and, after using that $(x_\e, v_\e, t_\e)$ is a minimum point of $J^\e-\phi^\e$, we get  
%\begin{equation}\label{takis1081}
%%\begin{split}
%	\phi^\e_t(x_\e,v_\e,t_\e) + a_{n,\tau} v_\e \cdot D \phi^\e(x_\e,v_\e,t_\e) + \frac{1}{\tau}
%		%\leq  J^\e_t + a_{n,\tau} v \cdot D J^\e\\
%%		&\geq \frac{1}{\tau}\left( \fint_{S^{n-1}} e^{\frac{1}{\e}( J^{\e}(x_\e,v',t_\e) - J^\e(x_\e,v_\e,t_\e))} dv' - 1\right)
%%			+ \fint_{S^{n-1}} e^{\frac{1}{\e}( J^\e(x_\e,v',t_\e) - J^\e(x_\e,v_\e,t_\e))} dv'.
%\geq \frac{\tau + 1}{\tau} \fint_{S^{n-1}} e^{\eta(v') -\eta(v_\e)} dv'. 
%			%+ \fint_{S^{n-1}} e^{\frac{1}{\e}( J^\e(x_\e,v',t_\e) - J^\e(x_\e,v_\e,t_\e))} dv',
%%			\\
%%		&\geq \frac{1}{\tau}\left( \fint_{S^{n-1}} e^{\eta(x_\e,v',t_\e) - \eta(x_\e,v_\e,t_\e)} dv' - 1\right)
%%			+ \fint_{S^{n-1}} e^{ \eta(x_\e,v',t_\e) - \eta(x_\e,v_\e,t_\e))} dv'(1 - \rho)_+.
%%\end{split}
%\end{equation}
%
%
Using now the definition of $\eta$ in \eqref{takis108} we find
\[ \phi_t^\e(x_\e,t_\e) + a_{n,\tau} v_\e \cdot D\phi^\e(x_\e,t_\e) \leq - H(D\phi(x_0,t_0)) + a_{n,\tau} v_\e\cdot D\phi(x_0,t_0),\]
which, after letting $\e \to 0$, gives \eqref{takis141}.
%the desired inequality
%\[ \phi_t(x_0,t_0) + H(D\phi(x_0,t_0)) + \frac{1}{\tau} \leq 0.\]
%%which is the desired inequality.
%%\smallskip
\smallskip

\textbf{Case two:} If 
$$ \Phi(1) \leq  \frac{\tau }{1+\tau} a_{n,\tau}|D\phi(x_0,t_0)|,$$
we observe that the previous argument cannot be repeated verbatim, since, as remarked on during the previous discussion, we do not have an exact corrector.  
We follow the same line of proof as above  but with an additional twist to deal with this difficulty.
% lack of corrector.
%fact that, in this case, the corrector $\eta$ given by \eqref{takis104} is not defined  when $v= D\phi(x_0,t_0)/|D\phi(x_0,t_0)|$. 
\smallskip

For $\delta>0$, we consider the ``approximate'' corrector $\eta_\delta$ given by 
%$$e^{\eta_\delta(v)} =\frac{\tau +1 }{\tau  a_{n,\tau}|p_0|} \frac{1}{\max (1 + v\cdot \hat p_0 ,\delta)}.$$
$$e^{\eta_\delta(v)} =\frac{\tau +1 }{\tau  a_{n,\tau}|p_0|} \frac{1}{1 + \delta + v\cdot \hat p_0},$$
%\smallskip
and note that 
\begin{equation}\label{takis160}
\begin{split}  
\fint_{S^{n-1}} e^{\eta_\delta(v)}dv& =\frac{\tau +1}{\tau a_{n,\tau}|p_0|} \fint_{S^{n-1}} \frac{1}{1 + \delta + v\cdot \hat p_0}dv\\
& \leq \frac{\tau +1}{\tau  a_{n,\tau}|p_0|} \fint_{S^{n-1}} \frac{1}{1 + v\cdot \hat p_0}
=  \frac{\tau +1}{\tau a_{n,\tau}|p_0|} \Phi(1) \leq 1.
\end{split}
\end{equation}

Consider the perturbed test function $\phi^{\delta,\e}(x,v,t)=\phi (x,t) +\e\eta_\delta(v)$. 
Let  $(x_{\delta,\e},t_{\delta,\e})$ be a maximum point  of $\max_vJ^\e_A - \phi^{\delta, \epsilon}$ in a neighborhood of $(x_0,t_0)$.  Then there exists $v_{\delta,\e} \in S^{n-1}$ such that $(x_{\delta,\e},v_{\delta,\e},t_{\delta,\e})$ is a maximum point of $J^\e_A -\phi^{\delta,\e}$ and  along a subsequence, which we denote the same way,  $\e\to 0$, $(x_{\delta,\e},v_{\delta,\e},t_{\delta,\e}) \to (x_0, v_{\delta,0}, t_0)$ for some $v_{\delta,0} \in S^{n-1}$ and $J^\e_A (x_{\delta,\e},v_{\delta,\e}, t_{\delta,\e}) \to J_A^\tau (x_0, t_0).$ Note that the limit of the $(x_{\delta,\e}, t_{\delta,\e})$ is independent of $\delta$ since $(x_0,t_0)$ is a strict local maximum of $\overline J_A -\phi$. 
\smallskip

As before we have
\[
\phi_t^{\delta,\e} (x_{\delta,\e},t_{\delta,\e}) + a_{n,\tau} v_{\delta,\e} \cdot D\phi^{\delta,\e}(x_{\delta,\e},t_{\delta,\e}) +\frac{1}{\tau}
		%\leq  J^\e_t + a_{n,\tau} v \cdot D J^\e\\
%		&\geq \frac{1}{\tau}\left( \fint_{S^{n-1}} e^{\frac{1}{\e}( J^{\e}(x_\e,v',t_\e) - J^\e(x_\e,v_\e,t_\e))} dv' - 1\right)
%			+ \fint_{S^{n-1}} e^{\frac{1}{\e}( J^\e(x_\e,v',t_\e) - J^\e(x_\e,v_\e,t_\e))} dv'.
\leq \frac{\tau +1}{\tau}\fint_{S^{n-1}} e^{\eta_\delta(v')-\eta_\delta(v_{\delta,\e})} dv'.
\]

Using the definition of $\eta_\delta$ and \eqref{takis160}, we find 
\[
\phi_t(x_{\delta,\e},t_{\delta,\e}) + a_{n,\tau} v_{\delta, \e} \cdot D\phi(x_{\delta,\e},t_{\delta,\e}) +\frac{1}{\tau} \leq a_{n,\tau}|p_0| (1+\delta+ v_{\delta,\e} \cdot \hat p_0)
	= a_{n,\tau} \left(|p_0|(1+\delta) + v_{\delta,\e} \cdot p_0\right).
\]
After letting $\e\to 0$, we obtain
\[
	\phi_t(x_0,t_0) + a_{n,\tau} v_{\delta, 0} \cdot D\phi(x_0,t_0) +\frac{1}{\tau} \leq a_{n,\tau}\left(|p_0|(1+\delta)+ v_{\delta,0} \cdot p_0\right),
\]
and, hence, 
\[
	\phi_t(x_0,t_0) - (1+\delta) a_{n,\tau} |D\phi(x_0,t_0)| +\frac{1}{\tau} \leq 0,
\]
from which the claim follows after taking $\delta \to 0$.
\medskip

\textbf{\# The proof of~\eqref{eq:J_A2} at $t=0$:}
To conclude the proof of the sub-solution property, we consider the case $t_0 = 0$. Here the first conclusion  is that $\overline J_A$ satisfies the initial condition in the following weak sense:
\beq\label{takis171}
\min\left\{ \max\left\{ \overline J_{A,t} + H(D\overline J_A),  \overline J_A, - \overline J_A - A \right\}, \overline J_A - ({{\bf 1}_A})^\star \right\} \leq 0 \ \ \text{on} \ \  \R^n\times\{0\}.
 \eeq
Then an argument as in \cite{BarlesEvansSouganidis} implies that actually 
$$\overline J_A \leq ({\bf 1}_A)^\star \ \ \text{on} \ \  \R^n\times\{0\},$$
and hence the claim.
\smallskip

To prove \eqref{takis171} we assume that, for some smooth $\phi$,  
%that $\underline J_A$ satisfies the correct initial conditions.  With this in mind, suppose that
$\overline J_A - \phi$ has a strict maximum at $(x_0,0)$.  Then we argue as before, that is we consider the perturbed test function we used for the proof when $t_0>0$. If the maximum points at the level $\e$ are in $\R^n\times\R_+$ for infinitely many $\e$'s, we argue exactly as before and obtain 
$$\max\{ \phi_t(x_0,t_0) + H(D\phi(x_0,t_0)), -\overline J_A(x_0,t_0), - \overline J_A(x_0,t_0) - A\} \leq 0,$$
otherwise we have
 $$\overline J_A(x_0,t_0) \leq   ({\bf 1}_A)^\star,$$
 and, hence, the claim.
 \end{proof}

%Then we find, as before, we consider the perturbed test function 
%$\phi^\e=\phi +\e\eta_\delta $v_\e \in S^{n-1}$ such that $(x_\e,v_\e,t_\e)$ is a maximum point of $J^\e_A -\phi^\e$ and  along a subsequence, which we denote the same way,  $\e\to 0$, there exists some $v_0\in S^{n-1}$ such that $(x_\e,v_\e, t_\e) \to (x_0, v_0, t_0)$ and $J^\e_A (x_\e,v_\e, t_\e) \to \overline J_A (x_0, t_0).$
% 
% If for 
%
%If $\Phi(1) > \tau a_{n,\tau} |p_0|/(1+\tau)$, define $\phi^\e = \phi + \epsilon \eta$ with $\eta$ as in~\eqref{takis101}.   As before, take a sequence of maxima of $J^\e_A - \psi^\e$, $(x_\e,v_\e, t_\e)$, that converges along a subsequence that we denote in the same way, $\e \to 0$, to $(x_0,v_0,0)$, for some $v_0 \in S^{n-1}$.
%\smallskip
%
%There are two cases to consider: either $t_\e = 0$ for infinitely many $\e$ or not.  In the first case, the result is clear due to the initial conditions of $J_0^\e$.  In the second case, we argue exactly as above to show that $\max\{ (\overline J_A)_t + H(D\overline J_A), \overline J_A, - A - \overline J_A\} \leq 0$.  
%\smallskip
%
%If $\Phi(1) \leq \tau a_{n,\tau} |p_0|/(1+\tau)$, the argument is a straightforward combination of the work above.  Hence, $\overline J_A$ satisfies~\eqref{eq:J_A2}.
%\smallskip
%

We now proceed with the proof of part (ii).

\begin{proof}[Proof of \Cref{lem:sub_super}(ii)]\ \smallskip \\
\textbf{\# The proof of~\eqref{eq:J_A} for positive times:}
Let $\phi$ be a smooth test function and assume that $\underline J_A - \phi$ has a strict minimum at $(x_0,t_0)$ and $\phi(x_0,t_0) = \underline J_A(x_0,t_0)$.
\smallskip

We claim that
\begin{equation}\label{chris1}
	\max\left\{ \phi_t(x_0,t_0) + H(D\phi(x_0,t_0) ),%\frac{\tau}{1+\tau} - \fint_{S^{n-1}}\frac{dv}{\frac{1}{\tau} + \phi_t + a_{n,\tau}v\cdot D \phi},
		\underline J_A(x_0,t_0), -  \underline J_A(x_0,t_0) - A \right\} \geq 0.
\end{equation}
If $\underline J_A(x_0,t_0) = 0$ or $\underline J_A(x_0,t_0) = -A$ the claim is true.  Hence
%, since $-A\leq \underline J_A \leq 0$,  
we assume that $-A <\underline J_A(x_0,t_0) < 0$ and  show that
\begin{equation}\label{chris2}
	\phi_t(x_0,t_0) + H(D\phi(x_0,t_0)) \geq 0.
\end{equation}
Again, to simplify the presentation and shorten some formulae, we use as before the notation $p_0$ and $\hat p_0$ for $D\phi(x_0,t_0)$ and $D\phi(x_0,t_0)/|D\phi(x_0,t_0)|$ respectively.   There are two cases to consider.
%in what follows we write
%\[
%	p_0 := D \phi(x_0,t_0)
%		~\text{ and }~
%		\hat p_0 := \frac{D \phi(x_0,t_0)}{|D\phi(x_0,t_0)|}.
%\]
%
\smallskip

\textbf{Case one:}  If  
$$ \Phi(1) >  \frac{\tau }{1+\tau} a_{n,\tau}|p_0|,$$
we use the perturbed test $\phi^\e=\phi + \e \eta(v; p_0)$ with $\eta$ given by \eqref{takis101}.
\smallskip

Let  $(x_\e,t_\e)$ be a minimum point  of $\min_vJ^\e_A - \phi^\epsilon$ in a neighborhood of $(x_0,t_0)$.  Then there exist $v_\e \in S^{n-1}$ such that $(x_\e,v_\e,t_\e)$ is a minimum  point of $J^\e_A -\phi^\e$ and,  along a subsequence, which we denote the same way,  $\e\to 0$, there exists $v_0\in S^{n-1}$ such that $(x_\e,v_\e, t_\e) \to (x_0, v_0, t_0)$ and $J^\e_A (x_\e,v_\e, t_\e) \to \overline J_A (x_0, t_0).$  We note that, for $\e$ sufficiently small, $-A < J^\e_A(x_\e, v_\e, t_\e) < 0$ and $J^\e_A(x_\e,v_\e,t_\e) = J^\e(x_\e,v_\e,t_\e)$, and, hence, the $J^\e(x_\e,v_\e,t_\e)$'s are bounded away from $0$.
\smallskip

Then \eqref{eq:the_equation} implies 
\begin{equation}\label{chris3}
\begin{split}
	\phi_t^\e(x_\e,v_\e,t_\e) + & a_{n,\tau} v_\e \cdot D \phi^\e (x_\e,v_\e,t_\e) +\frac{1}{\tau}\\
&\geq \left(\frac{1}{\tau} + \left(1 - \rho^\e(x_\e,t_\e)\right)_+\right)\fint_{S^{n-1}} e^{\frac{J^\e(x_\e,v',t_\e) - J(x_\e,v_\e,t_\e)}{\e}}dv'.
\end{split}
\end{equation}
Since the left hand side of \eqref{chris3} is bounded independently of $\e$, so must be  the integral term
\[
	\fint_{S^{n-1}} e^{(J^\e(x_\e,v',t_\e) - J(x_\e,v_\e,t_\e))/\e}dv'.\] % must be bounded independently of $\e$ as well.  
%\smallskip
%Moreover recall that $J^\e(v_\e,v_\e,t_\e)$ is negative and bounded away from zero. % (recall that they converge to $\underline J_A(x_0,t_0)$),
%\smallskip
Finally, note that
\[
	\rho^\e(x_\e,t_\e)
		= e^{\frac{J^\e(x_\e,v_\e,t_\e)}{\e}}\fint_{S^{n-1}}e^{\frac{J^\e(x_\e,v',t_\e)-J^\e(x_\e,v_\e,t_\e)}{\e}} dv'.
\]
%Since, $\fint_{S^{n-1}} e^{\eta(v')-\eta(v_\e)} dv'$ is bounded away from $0$ and the  $J_A^\tau(x_\e,v_\e,t_\e)$'s  are  negative and bounded away from zero (recall that they converge to $\underline J_A(x_0,t_0)$, we find that 
Combining the last two observations and the formula above we conclude that 
%Since $J^\e(v_\e,v_\e,t_\e)$ is negative and bounded away from zero (recall that they converge to $\underline J_A(x_0,t_0)$), we find that
\begin{equation}\label{takis180}
 \lim_{\e \to 0} \rho^\e(x_\e,t_\e) =0. 
\end{equation}
Returning to \eqref{chris3}, we obtain
\[
\begin{split}
	\phi_t^\e(x_\e,v_\e,t_\e) + & a_{n,\tau} v_\e \cdot D \phi^\e (x_\e,v_\e,t_\e) +\frac{1}{\tau}\\
&\geq \frac{1+\tau}{\tau}\fint_{S^{n-1}} %e^{\eta(v')-\eta(v_\e)} dv'. 
e^{\frac{J^\e(x_\e,v',t_\e)-J^\e(x_\e,v_\e,t_\e)}{\e}} dv' + \text{o}(1).
\end{split}
\]
Using now that $(x_\e,v_\e,t_\e)$ is a minimum point of $J_A^{\e} - \phi^\e$, we get
%Using that $\rho^\e(x_\e,t_\e) \to 0$ and that, for $\e$ small enough and all $v \in S^{n-1}$ and all $(x,t)$ in a small ball $B_r(x_0,t_0)$ of radius $r >0$, 
%\begin{equation}\label{eq:test_function_difference2}
%	\phi^\e(x,v,t) - \phi^\e(x_\e,v_\e, t_\e)
%		\leq J^\e_A(x,v,t) - J^\e_A(x_\e,v_\e,t_\e)
%		= J^\e(x,v,t) - J^\e(x_\e,v_\e,t_\e),
%\end{equation}
%we obtain, from~\eqref{chris3}
\begin{equation}\label{chris7}
	\phi^\e_t(x_\e,v_\e,t_\e) + a_{n,\tau} v_\e \cdot D \phi^\e(x_\e,v_\e,t_\e) +\frac{1}{\tau}
\geq \frac{1+\tau}{\tau}\fint_{S^{n-1}} e^{\eta(v')-\eta(v_\e)} dv' + \text{o}(1).
\end{equation}

The definition of $\eta$ in \eqref{takis101} then yields 
\[ \phi_t^\e(x_\e,t_\e) + a_{n,\tau} v_\e \cdot D\phi^\e(x_\e,t_\e)  \geq -H(D\phi(x_0,t_0)) + a_{n,\tau}v_\e\cdot D\phi(x_0,t_0) + \text{o}(1),\]
which, after letting $\e \to 0$, gives the desired inequality.
%\[ \phi_t(x_0,t_0) + H(D\phi(x_0,t_0)) + \frac{1}{\tau} \geq 0.\]
\smallskip

\textbf{Case two:} If  
$$ \Phi(1) \leq  \frac{\tau }{1+\tau} a_{n,\tau}|p_0|,$$ the argument needs to be modified.  We recall that this may only occur when $n \geq  4$.
\smallskip

We consider, for $\mu \in (0, \mu_c)$, with $\mu_c:= (n-1)/2$, and for $\delta > 0$, the ``approximate'' corrector
\begin{equation}\label{takis109}
\eta_{\mu,\delta} (v)= \mu \log\left(\frac{1}{1+\delta + v\cdot \hat p_0}\right).
\end{equation}
The perturbed test function is now $\phi^{\mu,\delta, \e}(x,v,t):=\phi (x,t) + \e \eta_{\mu,\delta} (v)$. 
%Let, as in the previous case, $(x_{\mu,\delta,\e}, v_{\mu,\delta,\e}, t_{\mu,\delta,\e})$ be local minima of $J^\e-\phi^{\mu,\delta,\e}$ such that $(x_{\mu,\delta,\e}, v_{\mu,\delta,\e}, t_{\mu,\delta,\e})$ converges along a subsequence, which we denote the same way, $\e \to 0$, to $(x_0,v_{\mu,\delta,0},t_0)$ for some $v_{\mu,\delta,0}\in S^{n-1}$, and $J^\e(x_{\mu,\delta,\e}, v_{\mu,\delta,\e}, t_{\mu,\delta,\e}) \to \underline J_A(x_0,v_{\mu,\delta,0},t_0).$  That $(x_{\delta,\e}, t_{\delta,\e}) \to (x_0, t_0)$ as $\e\to 0$ is a consequence of the fact that $(x_0,t_0)$ is a strict minimum of $\underline J_A-\phi.$
As in the previous case, let $(x_{\mu,\delta,\e}, t_{\mu,\delta,\e})$ be local minima of $\min_v J_A^\e-\phi^{\mu,\delta,\e}$.  Then there exists $v_{\mu,\delta,\e} \in S^{n-1}$ such that $(x_{\mu,\delta,\e}, v_{\mu,\delta,\e}, t_{\mu,\delta,\e})$ is a minimum point of $J_A^\e - \phi^{\mu,\delta,\e}$.  Further, along a subsequence, which we denote the same way, $\e \to 0$, there exists $v_{\mu,\delta,0}\in S^{n-1}$ such that $(x_{\mu,\delta,\e}, v_{\mu,\delta,\e}, t_{\mu,\delta,\e})$ converges to $(x_0,v_{\mu,\delta,0},t_0)$ and $J_A^\e(x_{\mu,\delta,\e}, v_{\mu,\delta,\e}, t_{\mu,\delta,\e}) \to \underline J_A(x_0,t_0).$  That $(x_{\delta,\e}, t_{\delta,\e}) \to (x_0, t_0)$ as $\e\to 0$ is a consequence of the fact that $(x_0,t_0)$ is a strict minimum of $\underline J_A-\phi.$ 
\smallskip

The fact that \eqref{takis180} holds is proved as before. Moreover,  as in~\eqref{chris7}, we find
\begin{equation}\label{chris4}
\begin{split}
	\phi^{\mu,\delta, \e}_t(x_{\mu,\delta, \e},v_{\mu,\delta, \e},t_{\mu,\delta, \e}) &+ a_{n,\tau} v_{\mu,\delta, \e} \cdot D \phi^{\mu,\delta, \e}(x_{\mu,\delta, \e},v_{\mu,\delta, \e},t_{\mu,\delta, \e}) + \frac{1}{\tau} + \text{o}(1)\\
		&\geq \frac{1+\tau}{\tau} \fint_{S^{n-1}} e^{ \eta_{\mu,\delta}(v') - \eta_{\mu,\delta}(v_{\mu,\delta, \e})} dv'.
			%+ \fint_{S^{n-1}} e^{\frac{1}{\e}( J^\e(x_\e,v',t_\e) - J^\e(x_\e,v_\e,t_\e))} dv'.
\end{split}
\end{equation}
Letting $\e \to 0$ in~\eqref{chris4} yields
\begin{equation}\label{chris5}
	\phi_t(x_0,t_0) + a_{n,\tau} v_{\mu,\delta,0} \cdot D \phi(x_0,t_0) + \frac{1}{\tau}
		\geq \frac{1+\tau}{\tau} \fint_{S^{n-1}} e^{ \eta_{\mu,\delta}(v') - \eta_{\mu,\delta}(v_{\mu,\delta,0})} dv'.
\end{equation}

Using now the definition of $\eta_{\mu,\delta}$, we rewrite \eqref{chris5} as 
\begin{equation}\label{chris6}
\begin{split}
	\phi_t(x_0,t_0) + a_{n,\tau} v_{\mu,\delta,0} \cdot D \phi(x_0,t_0) + \frac{1}{\tau}
		%\leq  J^\e_t + a_{n,\tau} v \cdot D J^\e\\
%&\geq a_{n,\tau}|D\phi (x_0,t_0)| \max\left(1 + v_{\delta,\e}\cdot \frac{D\phi(x_0,t_0)}{|D\phi(x_0,t_0)|},\delta \right)\fint_{S^{n-1}} e^{( \eta_{\delta}(v')}dv'.
	\geq \frac{1+\tau}{\tau}\left(1 + \delta + v_{\mu,\delta,0}\cdot \hat p_0 \right)^\mu \fint_{S^{n-1}} \frac{dv'}{\left(1 + \delta + v' \cdot \hat p_0\right)^\mu}.
\end{split}
\end{equation}
Along a subsequence, which we denote in the same way, $\delta \to0$, it follows that $v_{\mu,\delta,0} \to v_{\mu,0,0}$ for some $v_{\mu,0,0} \in S^{n-1}$, and 
\begin{equation}\label{chris8}
	\phi_t (x_0,t_0) + a_{n,\tau} v_{\mu,0,0} D\phi(x_0,t_0) + \frac{1}{\tau}
		\geq \frac{1+\tau}{\tau} \left(1 + v_{\mu,0,0} \cdot\hat p_0\right)^\mu \fint_{S^{n-1}} \frac{dv'}{(1 + v'\cdot \hat p_0)^\mu} dv'.
\end{equation}
It shown in the Appendix that 
\[
\lim_{\mu\to \mu_c} \fint_{S^{n-1}} \frac{dv'}{(1 + v'\cdot \hat p_0)^\mu}dv'=\infty,
\]
%It is easy to see that the integral on the right hand side tends to $+\infty$ as $\mu \ro \mu_c$; see the discussion in the appendix for details.  
while the left hand side of \eqref{chris8} is bounded independently of $\mu$.  It follows that $v_{\mu,0,0} \to - \hat p_0$ as $\mu \to \mu_c$.  
\smallskip

 Hence, %after taking this limit, we obtain
\[
	\phi_t (x_0,t_0) - a_{n,\tau} |D\phi(x_0,t_0)| + \frac{1}{\tau}
		\geq \liminf_{\mu\to\mu_c} \frac{1+\tau}{\tau} \left(1 + v_{\mu,0,0} \cdot \hat p_0\right)^\mu \fint_{S^{n-1}} \frac{dv'}{(1 + v'\cdot \hat p_0)^\mu}
		\geq 0,
\]
and the proof is now complete in the case that $t_0 = 0$.
\medskip

\textbf{\# The proof of~\eqref{eq:J_A} at $t=0$:}  The case $t_0=0$ is treated very similarly to the work above; see the proof of the sub-solution property.   As such, we omit it.  The proof is now complete.
\end{proof}

From the above, we now easily show that $J^\e$ converges uniformly to the solution of~\eqref{eq:J}. %not depending on $J$. 
 In other words, we establish \Cref{thm:phi_convergence}.

\begin{proof}[Proof of~\Cref{thm:phi_convergence}]
By construction, $\underline J_A \leq \overline J_A$.  Also, since $\underline J_A$ is a super-solution to an equation that $\overline J_A$ is a sub-solution to, it follows that $\underline J_A \geq \overline J_A$, by the comparison principle (see Crandall, Lions and Souganidis~\cite{CrandallLionsSouganidis}).  Hence, we have that $\underline J_A = \overline J_A$.  This, in turn, implies that $J^\e_A$ converges locally uniformly in $(x,t)$ and uniformly in $v$ to $J_A^\tau = \underline J_A = \overline J_A$, which does not depend on $v$ and which solves~\eqref{eq:J_A}.
\smallskip

On the other hand, it is easily seen that $\underline J_A = \inf\{-A, \underline J\}$ and $\overline J_A = \inf\{ - A, \overline J\}$.  Letting $A \to \infty$, we see that $J^\e$ converges locally uniformly to a limit $J$ which solves~\eqref{eq:J}.
\end{proof}

\section{The speed of the moving front}\label{sec:speeds}

\subsection*{Propagation speed when $n>1$}

\begin{proof}[Proof of \Cref{prop:speedn}]
%First, we show that $H(p)/|p|$ has a maximum.  
When $n>3$, the fact that $H(p)/|p|$ attains its maximum follows from the observation that, for large $|p|$,
$$\frac{H(p)}{|p|}=\frac{1}{\tau |p|} - a_{n,\tau} > -a_{n,\tau} =\lim_{|p|\to \infty}\frac{H(p)}{|p|}.$$
This last observation also yields that  $c_{n,\tau} <a_{n,\tau}.$
%\[
%	- c_{n,\tau} = \max_{p \in \R^n} \frac{H(p)}{|p|} > - a_{n,\tau}.
%\]
%Since  $H(0) = -1$, 
%\[
%	\lim_{p\to 0} \frac{H(p)}{|p|} = -\infty.
%\]
%When $n > 3$, by the definition of $H$,
%\begin{equation}\label{chris30}
%	\lim_{|p|\to\infty} \frac{H(p)}{|p|}
%		= \lim_{|p|\to\infty} \left( -a_{n,\tau} + \frac{1}{\tau|p|}\right)
%		= - a_{n,\tau}.
%\end{equation}
%Finally, since for  $|p|$ sufficiently large $|p|$,   $H(p)/|p|=1/\tau|p| -a_{n,\tau}$ is decreasing,  it follows that, when $n >3$, $H(p)/|p|$ attains its maximum, and, further, that
%\[
%	- c_{n,\tau} = \max_{p \in \R^n} \frac{H(p)}{|p|} > - a_{n,\tau}.
%\]
\smallskip

SWhen $n=2, 3$, $H$ is  defined  by the second alternative  in~\eqref{takis30} and the argument above can not be used. Instead we compute the Hamiltonians in order to conclude that $H(p)/|p|$ attains a maximum.
\smallskip

We first consider the case $n=2$.  Using \eqref{takis230},  \eqref{takis30} and the fact that, as shown in the Appendix, for $s\geq 1$,
\begin{equation}\label{takis250}
\Phi(s)=\frac{1}{\sqrt{s^2-1}},
\end{equation} 
we find 
\begin{equation}\label{eq:2d_Hamiltonian}
	H(p) = \frac{1}{\tau}- \sqrt{\left(\frac{1+\tau}{\tau}\right)^2 + \frac{2}{\tau} |p|^2}.
\end{equation}
%and the claim follows after some calculations. 

%Dividing by $|p|$ on both sides and applying Taylor's theorem, we obtain, as $|p| \to \infty$,
%\[
%	\frac{H(p)}{|p|}
%		= \frac{1}{\tau |p|} - \sqrt{\left(\frac{1+\tau}{\tau|p|}\right)^2 + \frac{2}{\tau}}
%		= \frac{1}{\tau |p|} - \sqrt{\frac{2}{\tau}} + O(|p|^{-2}).
%\]
%Clearly, $H(p)/|p| \to - a_{2,\tau}$.  Further, when $|p|$ is sufficiently large, $H(p)/|p| > - \sqrt{2/\tau} =-a_{2,\tau}$.  It follows that $H(p)/|p|$ attains its maximum and that
%\[
%	-c_{2,\tau} = \max_{p\in \R^2} \frac{H(p)}{|p|} > -a_{2,\tau}.
%\]
%
Next we consider the case $n=3$.  Using again  \eqref{takis230},  \eqref{takis30} and the fact that, as shown in the Appendix, 
for $s>1$,
\beq\label{takis251}
\Phi(s)=\frac{1}{2} \log\left( \frac{s + 1}{s-1}\right),
\eeq
%As above, we have that
%\[
%	\Phi \left( \frac{\alpha}{a_{3,\tau}|p|}\right)
%		= \frac{\tau}{1+\tau} a_{3,\tau} |p|,
%\]
%where, again, we let $\alpha = -H(p) + \tau^{-1}$.  Using the computation of $\Phi$ in \Cref{appendix} gives
%\[
%	\frac{1}{2} \log\left( \frac{\frac{\alpha}{a_{3,\tau}|p|} + 1}{\frac{\alpha}{a_{3,\tau}|p|}-1}\right)
%		= \frac{\tau}{1+\tau} a_{3,\tau} |p|.
%\]
%Exponentiating both sides, re-arranging the terms, and using the definition of $\alpha$, 
we obtain
\begin{equation}\label{eq:3d_Hamiltonian}
	H(p)
		= \frac{1}{\tau} - \frac{\sqrt 3 |p|}{\sqrt \tau} \frac{e^{\frac{2\sqrt{3\tau}}{1+\tau}|p|} + 1}{e^{\frac{2\sqrt{3\tau}}{1+\tau}|p|} - 1}.
\end{equation}
In either case the claim follows after some straightforward calculations. 

%Dividing by $|p|$ and using Taylor's theorem yields
%\[
%	\frac{H(p)}{|p|}
%		= \frac{1}{\tau |p|} - \frac{\sqrt 3}{\sqrt \tau} \frac{e^{\frac{2\sqrt{3\tau}}{1+\tau}|p|} + 1}{e^{\frac{2\sqrt{3\tau}}{1+\tau}|p|} - 1}
%		= \frac{1}{\tau |p|} - \frac{\sqrt 3}{\sqrt \tau} \frac{1 + e^{-\frac{2\sqrt{3\tau}}{1+\tau}|p|}}{1 - e^{-\frac{2\sqrt{3\tau}}{1+\tau}|p|}}
%		= \frac{1}{\tau |p|} - \frac{\sqrt 3}{\sqrt \tau} + O\left( e^{-\frac{2\sqrt{3\tau}}{1+\tau}|p|}\right),
%\]
%from which it follows that $\lim_{|p|\to\infty} H(p)/|p| = -\sqrt{3/\tau}.$  Furthermore, it is clear that $H(p)/|p| >  - \sqrt{3/\tau} = - a_{3,\tau}$ for $|p|$ sufficiently large.  We conclude that $H(p)/|p|$ has a maximum and that
%\[
%	-c_{3,\tau} = \max_{p\in \R^3} \frac{H(p)}{|p|} > -a_{3,\tau}.
%\]
\smallskip

In view of the discussion in the introduction, to conclude the proof it suffices to show that, for each $e\in S^{n-1}$, $L(c_{n,\tau}e)=0$.
\smallskip

%We have shown that $c_{n,\tau} := - \max H(p)/|p|$ is well-defined.  We now conclude \Cref{prop:speedn} by showing that $L(c)= 0$.  

Let $p_c \in \R^n\setminus\{0\}$ be a maximizing vector in the definition of $c_{n,\tau}$, that is,
$$c_{n,\tau} = -\frac{H(p_c)}{|p_c|};% \ \ \text{and} \ \  D_p H(p_c) = \frac{H(p_c)}{|p_c|} \hat p_c  = -c_{n,\tau} \hat p;
$$
% a vector such that $c_{n,\tau} = -H(p_c)/|p_c|$. Therefore, letting $\hat p_c = p_c/|p_c|$,
%\[
%	D_p H(p_c) = \frac{H(p_c)}{|p_c|} \hat p_c  = -c_{n,\tau} \hat p.
%\]
note that, in view of the isotropy of $H$,  any rotation of $p_c$ is also a maximizing vector. Thus  we may assume that $e \cdot p_c = -|p_c|$.  
\smallskip

Since, for $q\in \R^n$,
\[
	L(q) = \inf_{p \in \R^n}\left(q\cdot p - H(p)\right),
\]
it is immediate that
\[
	L(c_{n,\tau} e) \leq \left( - c_{n,\tau} |p_c| - H(p_c)\right)
		= 0.
\]
Rewriting the definition of $L$, we obtain
\begin{equation}\label{chris103}
L(c_{n,\tau} e)
		= c_{n,\tau}\inf_{p \in \R^n}|p| \left(e \cdot \hat p - \frac{1}{c_{n,\tau}}\frac{H(p)}{|p|}\right).
\end{equation}
%\begin{equation}\label{chris103}
%L(c_{n,\tau} e)
%		= c_{n,\tau} |p|\inf_{p \in \R^n}\left(e \cdot \hat p - \frac{1}{c_{n,\tau}}\frac{H(p)}{|p|}\right)=c_{n,\tau} |p|\min_{p \in \R^n}\left(e \cdot \hat p - \frac{1}{c_{n,\tau}}\frac{H(p)}{|p|}\right).
%\end{equation}
%The Hamiltonian $H$ is continuous and tends to $-a_{n,\tau}$ for $|p|$ sufficiently large.  This combined with the bound $c_{n,\tau} < a_{n,\tau}$ from above implies that the $\inf$ in~\eqref{chris103} may be replaced by $\min$.  Thus,
%\[
%	L(c_{n,\tau}) = c_{n,\tau}\min_{p \in \R^n}\left(\hat e \cdot p - \frac{1}{c_{n,\tau}}H(p)\right).
%\]

If $L(c_{n,\tau} e) < 0$, then there exists $p\in \R^n\setminus \{0\}$ such that
\[L(c_{n,\tau} e)\leq c_{n,\tau} |p| \left(-1- \frac{1}{c_{n,\tau}}\frac{H(p)}{|p|}\right) < 0.\]
%Also, $H$ is isotropic, so any minimizing vector $p$ must be parallel to $-\hat e$, giving
%\[
%	L(c_{n,\tau}) = c_{n,\tau}\min_{p \in \R^n}\left(-|p| - \frac{1}{c_{n,\tau}}H(p)\right).
%\]
%Using $p = p_c$ in the equation above, we see that $L(c_{n,\tau}) \leq 0$.  
This implies that
\[
	- c_{n,\tau} < \frac{H(p)}{|p|},
\]
which contradicts the definition of $c_{n,\tau}$. 
%\[
%	-c_{n,\tau} < \frac{H(p)}{|p|}.
%\]
%This contradicts the definition of $c_{n,\tau}$. 
 We conclude that $L(c_{n,\tau}e)=0$.
\end{proof}

%\smallskip
%
%{\bf From your re-written introduction, I would guess that you want to remove the last few paragraphs, which I'm fine with.  I didn't comment them out because I wasn't 100\% certain that you would agree.} Finally, to prove that $\partial \{J < 0 \} = \{x \in \R^n : d(x,G_0) = c_{n,\tau}t\}$, using~\eqref{chris35} and \eqref{chris36}, we need only show that if $|q| < c_{n,\tau}$ then $L(q) > 0$ and if $|q| > c_{n,\tau}$ then $L(q) < 0$.
%\smallskip
%
%If $|q| < c_{n,\tau}$, we suppose, by way of contradiction, that there exists $p$ such that $q\cdot p - H(p) \leq 0$.  Then, by the Cauchy-Schwarz inequality,
%\[
%	-|p||q|  - H(p) \leq p\cdot q - H(p) \leq 0.
%\]
%By assumption, $|q| < c_{n,\tau}$, which yields
%\[
%	-c_{n,\tau}|p| - H(p)
%		< -|p||q|  - H(p) \leq 0.
%\]
%Let $\hat q = q/|q|$ and $\tilde p = - \hat q |p|$.  Notice that $\hat q \cdot \tilde p = -|p|$.  %Using that $H$ is isotropic, we re-write the above as
%%\[
%%	q\cdot \tilde p - H(\tilde p) \leq 0.
%%\]
%From the definition of $L$ and the fact that $L(c_{n,\tau} \hat q) = 0$,
%\[
%	0 \leq c_{n,\tau} \hat q \cdot \tilde p - H(\tilde p)
%		= - c_{n,\tau} |p| - H(p)
%		<  0,
%\]
%which is clearly a contradiction.  We conclude that $L(q) > 0$ when $|q| < c_{n,\tau}$.
%\smallskip
%
%If $|q| > c_{n,\tau}$, take $p_c$ to be any vector such that $-c_{n,\tau} = H(p_c)/|p_c|$, as above.  Since $H$ is isotropic, we assume, without loss of generality, that $p_c \cdot q = -|p_c||q|$.  This yields
%\[
%	0 = -c_{n,\tau} |p_c| - H(p_c)
%		> -|q| |p_c| - H(p_c)
%		\geq L(q).
%\]
%This concludes the proof.
%\end{proof}

\subsection*{Explicit formulas for the speed for $n=1$ and $n=2$.} %one and two dimensions}

\begin{proof}[Proof of \Cref{prop:speed12}]
We begin with the speed when $n=2$ since the proof is simple.  Using \Cref{prop:speedn} along with~\eqref{eq:2d_Hamiltonian}, find 
\[
	c_{2,\tau}
		= - \max_{p\in\R^2} \frac{H(p)}{|p|}
		= - \max_{p\in\R^2} \left(\frac{1}{\tau |p|} - \sqrt{\left(\frac{1+\tau}{\tau|p|}\right)^2 + \frac{2}{\tau}}\right).
\]
Elementary calculus yields that  the maximum is attained when $|p| =  (1+\tau)\sqrt{2+\tau}/\sqrt{2}$, giving
\[
	c_{2,\tau} = \frac{\sqrt{2(2+\tau)}}{1+\tau}.
\]
%\smallskip

We now consider the case $n=1$. Let $\alpha = -H(p) +\tau^{-1}$. It follows from \eqref{takis30} and the fact that $S^0 = \{-1,1\}$ that, for any $p \in \R\setminus \{0\}$, 
%First, we compute the Hamiltonian. As noted above, $\Phi(1) = \infty$.  Thus, we need only consider the integral form of the Hamiltonian~\eqref{takis30}, given by~\eqref{eq:integral_form}.  For any $p \in \R^n\setminus \{0\}$,
%\[
%	\Phi\left(\frac{-H(p)+ \frac{1}{\tau}}{a_{1,\tau}|p|}\right) = \frac{\tau}{1+\tau}a_{1,\tau} |p|.
%\]
%Define $\alpha = -H(p) +\tau^{-1}$.  Since $S^0 = \{-1,1\}$,
\[
	\frac{\frac{\alpha}{a_{n,\tau}|p|}}{\frac{\alpha^2}{a_{1,\tau}^2 |p|^2} - 1} = \frac{\tau}{1+\tau}a_{1,\tau} |p|, 
\]
and after rearranging this %and completing the square yields
\[
	\alpha = \frac{1+\tau}{2\tau} + \sqrt{\left(\frac{1+\tau}{2\tau}\right)^2 + \frac{1}{\tau} |p|^2};
\]
that is
\begin{equation}\label{eq:1d_HJ}
	H(p)
		= \frac{1-\tau}{2\tau} - \sqrt{\left(\frac{1+\tau}{2\tau}\right)^2 + \frac{1}{\tau} |p|^2}.
\end{equation}
Using elementary calculus, it is easy to see that $c_{1,\tau}$ is given by the formula in~\Cref{prop:speed12}.
\smallskip

%Hence
%\begin{equation}\label{chris104}
%	\frac{H(p)}{|p|}
%		= \frac{1-\tau}{2\tau |p|} - \sqrt{\left(\frac{1+\tau}{2\tau|p|}\right)^2 + \frac{1}{\tau}}.
%\end{equation}
%There are two cases: $\tau < 1$ and $\tau \geq 1$.
%\smallskip

%If $\tau < 1$, the first term in~\eqref{chris104} is positive.  An argument as above shows that $H(p)/|p|$ attains a maximum.  By elementary calculus we see that the maximum occurs at $|p| = (1+\tau)/(1-\tau)$, which yields
%\[
%	c_{1,\tau} = \frac{2}{1+\tau}.
%\]
%\smallskip
The definition of the concave dual and elementary calculus now yields
\[
	L(q) =	\begin{cases}
		-\frac{1-\tau}{2\tau} + \frac{1+\tau}{2\tau}\sqrt{1 - \tau |q|^2}, \qquad &\text{ if } \tau |q|^2 \leq 1,\\
		-\infty, &\text{ otherwise}.
	\end{cases}
\]
\smallskip

If $\tau < 1$, $L(q) = 0$ if and only if $|q| = 2/(1+\tau)$.  In view of~\eqref{chris36}, this implies that $c_{1,\tau} = 2/(1+\tau)$.
\smallskip

If $\tau \geq 1$, there is no $q \in \R$ such that $L(q) = 0$.  In this case, we notice that $L(q) > 0$ when $\tau |q|^2 \leq 1$ and $L(q) < 0$ otherwise.  Hence $\partial\{q \in \R: L(q) < 0\} = \{q \in \R: |q| = 1/\sqrt\tau\}$.  In view of~\eqref{chris105}, this implies that $c_{1,\tau} = 1/\sqrt\tau$, finishing the proof.
\end{proof}

\section{The  limit $\tau \to 0$}\label{sec:tau}

Let  $J^\tau$ be  the solution to~\eqref{eq:J}, recall that   $J^\tau \leq 0$, and define
\[
	J_A^\tau := \max\{J^\tau, -A\}. 
\]
Consider the  half-relaxed limits
\begin{equation}\label{chris10}
	\overline z_A := \limsup_{\tau \to 0} J_A^\tau
		\qquad \text{ and }\qquad
		\underline z_A := \liminf_{\tau \to 0} J_A^\tau,
\end{equation}
%The change in notation, from $J$ to $z$, is to match the notation for the limiting equation. 
and note that, by construction,
\begin{equation}\label{chris100}
	-A \leq \underline z_A \leq \overline z_A \leq 0.
\end{equation}
\smallskip

We prove the following lemma.
\begin{lemma}\label{lem:tau_half_relaxed}
	Let $J^\tau$ solve~\eqref{eq:J} and define $\overline z_A$ and $\underline z_A$ by~\eqref{chris10}. Then\\
	(i) $\underline z_A$ is a super-solution to
	\begin{equation}\label{eq:z_A}
		%\begin{cases}
		\max\left\{ \underline z_{A,t} - |D \underline z_A|^2 -1, \underline z_A, -A-\underline z_A\right\} \geq 0 \   \text{ in } \  \R^n \times \R_+ \quad   \underline z_A(\cdot, 0) \geq ({{\bf 1}_A})_\star \  \text{ on } \ \R^n.
		%\end{cases}
	\end{equation}
	(ii)  $\overline z_A$ is a sub-solution to
	\begin{equation}\label{eq:z_A2}
		%\begin{cases}
		\max\left\{ \overline z_{A,t} - |D \overline z_A|^2 -1, \overline z_A, -A-\overline z_A\right\} \leq 0 \ \text{ in } \ \R^n \times \R_+, \quad 
		\overline z_A(\cdot, 0) \leq ({{\bf 1}_A})^\star  \ \text{ on } \ \R^n.
		%\end{cases}
	\end{equation}
\end{lemma}
We momentarily postpone the proof of Lemma~\ref{lem:tau_half_relaxed} to note  that the comparison principle yields  $\overline z_A = \underline z_A$. Letting then   $A\to \infty$ we obtain \Cref{prop:tau} from \Cref{lem:tau_half_relaxed}.  The argument follows exactly as in \Cref{lem:sub_super},  so we omit it.

\begin{proof}[Proof of \Cref{lem:tau_half_relaxed}]
We prove only the claim for $\underline z_A$.  The argument is similar for $\overline z_A$ with similar modifications as in the proof of \Cref{lem:sub_super}.  
\smallskip

Fix any smooth test function $\phi$ and any point $(x_0,t_0)$ such that $\underline z_A - \psi$ has a strict local minimum of zero at $(x_0,t_0)$.
\smallskip

Assume that $t_0>0$.  If $\underline z_A(x_0,t_0) = 0$ or $-A$ the claim is immediate from~\eqref{eq:z_A}. Hence,  %finished.  By~\eqref{chris100}, $-A \leq \underline z_A \leq 0$; thus, 
we need only consider the case where $\underline z_A(x_0,t_0) \in (-A,0)$ and  aim to show that
\[
	0 \leq \phi_t(x_0,t_0) - |D \phi(x_0,t_0)|^2 - 1.
\]
%\smallskip
It follows from  the definition of $\underline z_A$  that, as $\tau \to 0$,  there exists a sequence $(x_\tau, t_\tau) \to  (x_0,t_0)$ such that $J^{\tau}_A - \phi$ has a local minimum at $(x_\tau,t_\tau)$ and  $J_A^\tau(x_\tau,t_\tau) - \phi(x_\tau,t_\tau) \to \underline z_A(x_0,t_0) - \phi(x_0,t_0)$. % as $\tau \to 0$. 
In addition, since  $\underline z_A(x_0,t_0) \in (-A,0)$, $J_A^\tau(x_\tau,t_\tau) \in (-A,0)$ for all $\tau$ sufficiently small.  
\smallskip

Using the definition of $H^\tau$ and that $J^\tau$ satisfies~\eqref{eq:J}, when $\tau$ is sufficiently small, we obtain, at $(x_\tau,t_\tau)$,
\[
	\frac{\tau}{1+\tau} - \fint_{S^{n-1}} \frac{dv}{\frac{1}{\tau} + \phi_t(x_\tau,t_\tau) + a_{n,\tau} v \cdot D \phi(x_\tau,t_\tau)} \geq 0.
\]
We point out that we need only use this form of the Hamiltonian $H$~\eqref{takis30} since, for any fixed $p$, $\Phi(1) > \sqrt{n\tau}|p|/(1+\tau)$ when $\tau$ is small enough. %depending only on $|p|$.
\smallskip

Dividing both sides by $\tau$ and using Taylor's theorem on the integrand, we see that, at $(x_\tau,t_\tau)$,
\[
	0 \leq \frac{1}{1+\tau}
		- \fint \left(1 - \tau \phi_t(x_\tau,t_\tau) - \sqrt{\tau n} v\cdot D \phi(x_\tau,t_\tau)
			+ n \tau \left(v\cdot D \phi(x_\tau,t_\tau) \right)^2 + \text{o}(\tau)\right).
\]
The third term under the integral vanishes after integrating over $S^{n-1}$ as it is odd in $v$, while, a calculation in the Appendix,  yields that %the fourth term under the integral integrates to $\tau |D \phi(x_\tau,t_\tau)|^2$ since
\begin{equation}\label{eq:v_squared}
	\fint_{S^{n-1}} n (v\cdot D \phi(x_\tau,t_\tau))^2 dv = |D \phi(x_\tau,t_\tau)|^2.
\end{equation}
%This integral is non-standard, so we show the computation in~\Cref{appendix}.  
Applying Taylor's theorem again, it is now immediate that, at $(x_\tau,t_\tau)$,
\[
	0 \leq (1 - \tau + \text{o}(\tau)) - 1 + \tau \phi_t(x_\tau,t_\tau) - \tau |D \phi(x_\tau,t_\tau)|^2 + o(\tau),
\]
which, after dividing by $\tau$ and letting $\tau \to 0$, yields
\[
	0 \leq \phi_t(x_0,t_0) - |D \phi(x_0,t_0)|^2 - 1.
\]
The case when $t_0=0$ may be easily handled by a combination of the methods above and in \Cref{sec:proofs}.
\end{proof}

\section{The failure of positivity}\label{sec:positivity}

\subsection*{The reactive-telegraph equation: proof of~\Cref{thm:non_positive_telegraph} when $n = 1$}

We prove that the reactive-telegraph equation preserves positivity when $n=1$.  
We note that the upper bound in the statement of \Cref{thm:non_positive_telegraph} is clearly not sharp.  Due to the lack of maximum principle and the structure of the equation~\eqref{eq:reactive_telegraph}, the upper and lower bounds must be obtained simultaneously.  Since our main interest is in the lower bound, we do not optimize the  proof of the upper bound and use $2$ since it is sufficient to obtain the desired lower bound.

\begin{proof}[Proof of \Cref{thm:non_positive_telegraph} when $n=1$]
Before beginning we note that it is enough to prove the claim with nonlinearity  $ \rho(1-\rho)$, since $\rho(1-\rho) \leq \rho(1-\rho)_+$.  The key tool is the equivalence between~\eqref{eq:reactive_telegraph} and~\eqref{eq:the_prelim_equation} in one spatial dimension, which may also be written as~\eqref{eq:kinetic_system1}.  That is, let $p^\pm$ solve~\eqref{eq:kinetic_system1} with $p^+(\cdot,0) = p^-(\cdot,0) = \rho_0$, then $\rho = (p^+ + p^-)/2$.  Hence, if we show that $0 \leq p^+, p^- \leq 2$, the result follows for $\rho$.
\smallskip

%We note that there is nothing to prove when $\tau \leq 1$.  Indeed, due to \Cref{lem:a_priori}, $\rho(x,t) \leq 1$ for all $(x,t)\in \R\times [0,\infty)$.  Hence, the right hand side of~

Fix $\epsilon>0$.  By approximation, we may assume that $p^\pm$ are smooth and uniformly equal to $2\sqrt \e$ outside a compact set and that the initial data satisfy $2\sqrt\e \leq \rho_0 \leq 1$.  We define
\[
	p^+_\e(x,t) = p^+(x,t) + \e t \qquad \text{ and } \qquad
	p^-_\e(x,t) = p^-(x,t) + \e t,
\]
and let
\[
	T_\e\
		= \sup\{t\in(0,1/\sqrt\e) : \sqrt \e < p^+_\e(x,s), p^-_\e(x,s) < 2 - \sqrt \e \text{ for all } s \in (0,t)\}.
\]
Here, $T_\e$ is the first time that either $p^+_\e$ or $p^-_\e$ ``touches'' $\sqrt \e$ or $2-\sqrt \e$.  It is well-defined and positive due to the smoothness of $\rho_0$ that is inherited by $p^+_\e$ and $p^-_\e$.
\smallskip

Our goal is to show that $T_\e = 1/\sqrt \e$ for every $\e>0$.  Once we have shown this, the bounds on $p^+$ and $p^-$ follow by taking $\e\to0$.  To this end, we proceed by contradiction and assume that $T_\e \in (0,1/\sqrt\e)$.
\smallskip

Let $x_\e\in\R$ be such that $p^+_\e(x_\e, T_\e)$ is either $\sqrt \e$ or $2-\sqrt \e$. The argument follows similarly if $p^-_\e(x_\e,T_\e)$ is either $\sqrt \e$ or $2-\sqrt \e$.
\smallskip

Before continuing with the proof, we briefly justify the existence of $x_\e$.  Let $R_\e>0$ be sufficiently large so that $\rho_0\equiv 2\sqrt\e$ on $B_{R_\e}^c$.  Since \eqref{eq:kinetic_system1} has no space or time dependence, the finite speed of propagation for kinetic equations implies that $p^\pm_x \equiv 0$ on $B_{R_\e+(\tau \e)^{-1/2}}^c \times [0,1/\sqrt \e]$.  As such,~\eqref{eq:kinetic_system1} reduces to an ordinary differential equation in $t$ at every point $x\in B_{R_\e+(\tau\e)^{-1/2}}^c$.  The ODE is the same for both $p^+$ and $p^-$, and, hence,  the $p^+$ and $p^-$ are equal on $B_{R_\e+(\tau\e)^{-1/2}}^c$.  This, in turn, implies that $p^+ = p^- = \rho$ on $B_{R_\e+(\tau\e)^{-1/2}}^c$, which, from~\eqref{eq:kinetic_system1}, implies that
\[
	\rho_t = \rho(1-\rho) \qquad\text{ with }\qquad \rho_0(x) = 2\sqrt\e
\]
on $B_{R_\e+(\tau\e)^{-1/2}}^c\times[0,1/\sqrt\e]$.  At this point, it is clear that $p_\e^+$ and $p_\e^-$ can achieve neither the value $\sqrt \e$ nor the value $2-\sqrt \e$ on $B_{R_\e+(\tau\e)^{-1/2}}^c\times [0,1/\sqrt\e]$, so long as $\e $ is sufficiently small.  As such, if $T_\e < 1/\sqrt\e$, the smoothness of $p_\e^+$ and $p_\e^-$ guarantees the existence of $x_\e \in B_{R_\e + (\tau\e)^{-1/2}}$.
\smallskip

First we consider the case that $p^+_\e(x_\e,T_\e) = \sqrt \e$.  Using that $p^+_\e$ is smooth and  $(x_\e,T_\e)$ is the location of a minimum of $p^+_\e$ on $\R\times (0,T_\e]$, we see that, at $(x_\e,T_\e)$,
\[
	(p^+_\e)_t + \frac{1}{\sqrt\tau} (p^+_\e)_x
		\leq 0.
\]
From this, along with~\eqref{eq:kinetic_system1} and the definition of $p^+_\e$ in terms of $p^+$, we get that, at $(x_\e,T_\e)$,
\begin{equation}\label{eq:p1}
	\e + \frac{1}{\tau} \left(\rho - p^+\right) + \rho(1-\rho)
		= (p^+_\e)_t + \frac{1}{\sqrt\tau} (p^+_\e)_x
		\leq 0.
\end{equation}
The definition of $T_\e$ yields  that $\sqrt \e \leq p^-_\e(x_\e,T_\e) \leq 2 - \sqrt \e$.  Using these bounds for $p^-_\e$, the facts that $p^+_\e(x_\e,T_\e) = \sqrt \e$ and $\e T_\e < \sqrt \e$, and the relationship between $p^\pm_\e$ and $p^\pm$, we obtain
\[
	0
		\leq\rho(x_\e,T_\e)
		\leq 1.
\]
Also, since $p^+_\e(x_\e,T_\e) \leq p^-_\e(x_\e,T_\e)$, it follows that $p^+(x_\e,T_\e) \leq \rho(x_\e,T_\e)$.  Hence, at $(x_\e,T_\e)$,
\[
	\frac{1}{\tau} \left(\rho - p^+\right) + \rho(1-\rho) \geq 0.
\]
The combination of this inequality with the positivity of $\e$ contradicts~\eqref{eq:p1}, finishing this case.
\smallskip

Now we consider the case that $p^+_\e(x_\e,T_\e) = 2-\sqrt \e$.  Reasoning as above, we see that, at $(x_\e,T_\e)$,
\[
	0 \leq (p^+_\e)_t + \frac{1}{\sqrt \tau}(p^+_\e)_x
		= \e + \frac{1}{\tau}\left(\rho-p^+\right) + \rho(1-\rho).
\]
Moving the $\e$ and $p^+$ terms from the right hand side to the left hand side and using that $p^+ = p^+_\e -\e T_\e$ and  $T_\e < 1/\sqrt\e$, we obtain
\begin{equation}\label{eq:p2}
	\frac{2}{\tau} - \frac{\sqrt \e\left(2 - \tau \sqrt \e\right)}{\tau}	
		= \frac{1}{\tau}\left(2 - 2\sqrt \e\right) - \e	
		\leq \frac{1}{\tau}\left(p^+_\e - \e T_\e\right) - \e
		\leq\frac{1}{\tau}\rho + \rho(1-\rho).
\end{equation}
%more simply, 
%\begin{equation}\label{eq:p2}
%	\frac{2}{\tau} - \frac{\sqrt \e\left(2 - \tau \sqrt \e\right)}{\tau}	
%		\leq\frac{1}{\tau}\rho + \rho(1-\rho).
%\end{equation}
\smallskip

On the other hand the definition of $T_\e$ implies  that $\sqrt \e \leq p^-_\e(x_\e,T_\e) \leq 2-\sqrt \e$.  From this, we see that
\[
	1 - \sqrt \e
		\leq \frac{\left(p^+_\e(x_\e,T_\e) - \e T_\e\right) + \left(p^-_\e(x_\e,T_\e) - T_\e\right)}{2}
		= \rho(x_\e,T_\e)
		< \frac{p^+_\e(x_\e,T_\e) + p^-_\e(x_\e,T_\e)}{2}
		\leq 2 - \sqrt \e, 
\]
or, more succinctly,
\begin{equation}\label{eq:p3}
	1 - \sqrt \e
		\leq \rho(x_\e,T_\e)
		< 2 - \sqrt \e.
\end{equation}

When $\e$ is sufficiently small, depending only on $\tau$, the two inequalities~\eqref{eq:p2} and~\eqref{eq:p3} are incompatible.  Briefly, this may be seen by considering two cases.  First, if $\rho \in[1-\e, 3/2]$, we obtain
\[
	\frac{2}{\tau} - \frac{\sqrt \e\left(2 - \tau \sqrt \e\right)}{\tau}	
		\leq\frac{1}{\tau}\rho + \rho(1-\rho)
		\leq \frac{3}{2\tau} + \frac{3}{2} \sqrt \e.
\]
from~\eqref{eq:p2}.  This is clearly a contradiction when $\e$ is sufficiently small.  Second, if $\rho\in [3/2, 2-\sqrt \e]$, we find from~\eqref{eq:p2}
\[
	\frac{2}{\tau} - \frac{\sqrt \e\left(2 - \tau \sqrt \e\right)}{\tau}	
		\leq\frac{1}{\tau}\rho + \rho(1-\rho)
		\leq \frac{2-\sqrt \e}{\tau} + \frac{3}{2}\left(1 - \frac{3}{2}\right),
\]
which is also a contradiction for $\e$ is sufficiently small.
\smallskip

This finishes the consideration of the second case and thus finishes the proof.
\end{proof}

\subsection*{The reactive-telegraph equation: proof of~\Cref{thm:non_positive_telegraph} when $n \geq 2$}

We need the following estimate: %A key estimate that we use in the proof of \Cref{thm:non_positive_telegraph} is the following:
\begin{theorem}[Chapter 1, Theorem 4.1 \cite{Sogge}]\label{thm:hyperbolic_regularity}
Let $F_{\rm wave}: \R^2 \to \R$ be a smooth function such that $F_{\rm wave}(0,0) = 0$.  Assume that $2s \geq n + 2$, $f \in H^{s+1}(\R^n)$, and $g \in H^s(\R^n)$.  Then there exists $T>0$, depending on $F_{\rm wave}$, $\|f\|_{H^{s+1}}$, and $\|g\|_{H^s}$, such that   the initial value problem  
\[
	\begin{cases}
		u_{tt} = \Delta u + F_{\rm wave}(u,u_t) \ \ &\text{ in }  \ \ \R^n \times (0,T),\\
		u = f,\ \ \ u_t = g  \ \ &\text{ on } \ \ \R^n \times \{0\},
	\end{cases}
\]
has a unique solution $u$ satisfying, for some universal constant $C$, 
\[
	\|u\|_{H^{s+1}(\R^n\times [0,T])}\leq C\left( \|f\|_{H^{s+1}(\R^n)} + \|g\|_{H^s(\R^n)} \right).
\]
%where $C$ is a universal constant.
\end{theorem}

With this in hand, we now show that the reactive-telegraph equation~\eqref{eq:reactive_telegraph} does not preserve positivity.

\begin{proof}[Proof of \Cref{thm:non_positive_telegraph} when $n\geq 2$]
It is enough to work with $n=2$,  since any example in this setting also works for $n>2$.
%\textbf{We first make two reductions.  First, it is enough to work with $n=2$ since any example in this setting also works for $n>2$.  Second, we constuct initial data $\rho_0$ that becomes negative below for the non-linearity $F$ given by~\eqref{takis210}, and this solution the example we construct below remains bounded above by $1$ for all time.  Hence, were we to use the same initial in the equation with the non-linearity $F$ given by~\eqref{takis210}}
%% First we observe that we may reduce the proof of \Cref{thm:non_positive_telegraph} to the case $n =2$.  Our strategy is to ``forget'' the extra variables.  Indeed, fix $n > 2$, and suppose that $\rho$ satisfies~\eqref{eq:reactive_telegraph} in $\R^2\times [0,T)$, for some $T>0$.  Then defining, for $(x,t) \in \R^n \times [0,T)$,
%%\[
%%	\tilde \rho(x,t) := \rho(x_1,x_2,t),
%%\]
%%we see that $\tilde \rho$ satisfies~\eqref{eq:reactive_telegraph} in $\R^n \times [0,T)$.  Hence, we need only construct the initial data $\rho_0$ in $\R^2$ to conclude the general case.
\smallskip

%We now proceed with the proof in the case when $n=2$.  
Fix $\rho_0(x) = \epsilon e^{-|x|^2/\delta}$ for $\epsilon, \delta\in(0,1)$ so that   $\e \leq \delta^{5/4}$ and let   $\rho$ be the solution to 
\beq\label{takis300}
\begin{cases}	\tau \rho_{tt} + (1 - \tau + 2\tau\rho) \rho_t = \Delta\rho + \rho(1-\rho) \ \ \text{in} \ \ \R^2\times (0,T),\\[1mm]
\rho(\cdot, 0) = \rho_0 \ \ \text{and}  \ \ \rho_t(\cdot,0) = 0 \ \ \text{on} \ \ \R^2,\end{cases}
\eeq
where $T>0$ gives the time interval of existence, which we obtain below.
\smallskip

%$\rho(0,x) = \rho_0(x)$, and $\rho_t(0,x) = 0$.  
Next we change variables to transform \eqref{takis300}  to a wave equation with a right hand side which we can analyze using Theorem~\ref{thm:hyperbolic_regularity}. 
%and remove integrating factors to obtain an equivalent problem involving only the wave equation on the left hand side.
\smallskip

It is immediate that  $u(x,t) = e^{(1 - \tau) t/2\tau} \rho(x,t)$ satisfies 
\[
	\tau u_{tt} - \Delta u = u\left( 1 - \frac{(1-\tau)^2}{4\tau^2} + 2\tau e^{-(1-\tau)t/2\tau} u_t - e^{-(1-\tau) t/2\tau}u\right).
\]
Letting $v(x,t) = u(x,\sqrt{\tau}t)$, we find 
\[\begin{cases}
	v_{tt} - \Delta v = v\left( 1 - \frac{(1-\tau)^2}{4\tau^2} + 2\sqrt{\tau} e^{-(1-\tau)t/2\tau} v_t - e^{-(1-\tau) t/2\tau}v\right) \ \ \text{ on } \ \ \R^2\times (0,T),\\
	v(\cdot,0) = \rho_0 \ \ \text{and} \ \  v_t (\cdot,0)= \frac{1-\tau}{2\sqrt{\tau}} \rho_0 \ \  \text{ on } \ \  \R^2. %\times\{0\},
\end{cases}\]
\smallskip

Using  \Cref{thm:hyperbolic_regularity} with  $s = 7/2$ and $s =5/2$ respectively, we find $T > 0$ and $C_{\tau}$, which is independent of $T$ and will change from line to line below,  such that 
%It is well-known~\cite[Theorem~4.1]{Sogge} that %there exists a solution $v\in C^2$ to the above equation on a small interval $[0,T]$ and that
%there exist constants $C_{\tau}$ and $C_{\tau,\delta}$, the first depending only on $\tau$ and the second depending only on $\tau$ and $\delta$, such that
\begin{equation}\label{eq:wave1}
	\|v_t\|_{L^\infty((0,T)\times \R^2)}
		\leq C_{\rm Sob} \|v\|_{H^{7/2}((0,T)\times \R^2)}
		\leq C_\tau \|\rho_0\|_{H^{7/2}(\R^2)}
		\leq C_\tau \e \delta^{-5/4}
\end{equation}
and
\begin{equation}\label{eq:wave2}
	\|v\|_{L^\infty((0,T)\times \R^2)}
		\leq C_{\rm Sob} \|v\|_{H^{5/2}((0,T)\times\R^2)}
		\leq C_{\tau} \|\rho_0\|_{H^{5/2}(\R^2))}
		\leq C_{\tau} \epsilon \delta^{-3/4};
\end{equation}
note that in the second inequality we absorbed $C_{\rm Sob}$ into $C_{\tau}$ to simplify the notation.  
\smallskip

Since $\e \leq \delta^{5/4}$, it follows that, for some $C_\tau$ depending only on $\tau$, % that 
$$\|\rho_0\|_{H^{7/2}} \leq C_\tau.$$
In view of this and~\Cref{thm:hyperbolic_regularity}, it follows that $T$ does not depend on $\e$ or $\delta$. 
% for some, possibly new, $C_\tau$ depending only on $\tau$.  Hence,  $T$ does not depend on $\e$ or $\delta$.
\smallskip

% such that $v \in C^2$, we have that $\|v\|_{L^\infty((0,T)\times\R^3)} \leq C \epsilon$ for some constant $C_\delta$ depending only on $\delta$, $r$, and $\tau$.

Duhamel's formula for the solution to the wave equation for $n=2$ gives
%Using the general form of the solution to the wave equation in $\R^2$, we obtain
\begin{equation}\label{takis301}
\begin{split}
	&v(x,t) = \frac{1}{2}\fint_{B_t(x)} \frac{1}{\sqrt{t^2 - |x-y|^2}}\left[t^2\frac{1-\tau}{2\sqrt{n\tau}} \rho_0(y) + t\rho_0(y) + tD\rho_0(y) \cdot(y-x)\right] dy\\
		&\quad + \int_0^t s^2\fint_{B_s(x)} \frac{v}{\sqrt{s^2 - |x-y|^2}}\left( 1 - \frac{(1-\tau)^2}{4\tau^2} + 2\sqrt{\tau} e^{-(1-\tau)t/2\tau} v_t - e^{-(1-\tau) t/2\tau}v\right)dyds.
\end{split}
\end{equation}
%\[\begin{split}
%	v(x,t) &= \fint_{\partial B_t(x)} \left[t\frac{1-\tau}{2\sqrt{n\tau}} \rho_0(y) + \rho_0(y) + D\rho_0(y) \cdot(y-x)\right] dy\\
%		&\quad + \int_0^t s\fint_{\partial B_s(x)} v\left( 1 - \frac{(1-\tau)^2}{4\tau^2} + 2\sqrt{\tau} e^{-(1-\tau)t/2\tau} v_t - e^{-(1-\tau) t/2\tau}v\right)dyds.
%\end{split}\]
At this point we mention that this example does not work in one dimension because the general form of the solution is different in one dimension.  In particular, the crucial term $D\rho_0$ does not appear.
\smallskip

Inserting the bounds~\eqref{eq:wave1} and \eqref{eq:wave2} in \eqref{takis301},  we find 
\[
	v(0,t)
		\leq \epsilon\left(e^{-t^2/\delta}\left[ t \frac{|1- \tau|}{2\sqrt{\tau}} + 1 - 2\frac{t^2}{\delta}\right] + C_\tau \delta^{-3/4} t^2\left[1 + \epsilon \delta^{-5/4}\right]\right),
\]
which, for $\delta = t^2$,  becomes
\[
	v(0,t)
		\leq \epsilon\left( e^{-1}t \frac{|1- \tau|}{2\sqrt{\tau}} - 2e^{-1} + C_{\tau} t^{1/2} + \epsilon C_{\tau} t^{-2}\right).
\]
%where we have re-written $C_{\tau,\delta}$ as $C_{\tau,t}$ to indicate that by choosing $\delta$ to depend on $t$, the dependence of this constant changes as well.
The claim follows by choosing first $t$ and then $\epsilon$ small to obtain $v(0,t) < 0$, and, hence,    that $\rho(0,t/\sqrt\tau) < 0$. % as desired.  This concludes the proof.
\end{proof}

\subsection*{A kinetic model that does not preserve positivity}\label{sec:non_positivity_kinetic}

In the interest of simplicity, we study here  a slightly different model than~\eqref{eq:the_prelim_equation} and somewhat singular initial data. Generalizations to a larger class of equations and initial data are conceptually straightforward, though with significantly more involved computations.  
\smallskip

We consider solutions to a two-dimensional discrete version of~\eqref{eq:the_prelim_equation} given by
\begin{equation}\label{eq:discrete_kinetic}
	\begin{cases}
		p_t + a_{2,\tau} v \cdot D p = \frac{1}{\tau} \left(\rho - p\right) + p(1-p)  \ \  \text{ in } \ \ \R^2 \times S^1_d \times \R_+,\\[1mm]
		p(\cdot,0) = p_0 \ \  \text{ on } \ \ \R^2 \times S^1_d,% \times\{0\},
	\end{cases}
\end{equation}
where $S^1_d := \{e_1, e_2, -e_1, -e_2\}$ with $e_1, e_2$ the standard basis vectors of $\R^2$ and 
$$\rho(x,t) := \frac{1}{4}\left(\sum_{v \in S^1_d} p(x,v,t)\right).$$
We have:

\begin{prop}\label{prop:kinetic_negative}
	Fix $\tau > 4$.  There exist a bounded $p_0$ with $\left(\sum_{v\in S^1_d} p_0\right)/4 \leq 1$ and $(x_0,v_0,t_0) \in \R^2 \times S^1_d \times \R_+$ such that,  %$0 \leq p_0 \leq 1$ and 
%$\left(\sum_{v\in S^1_d} p_0\right)/4 \leq 1$ and a point $(x_0,v_0,t_0) \ $ 
if $p$ solves ~\eqref{eq:discrete_kinetic} with initial datum $p_0$,  then  $p(x_0,v_0,t_0) < 0$.
\end{prop}

The idea is to choose $p_0$ having  three patches, one moving in the  $-e_1$ direction, one moving in the $e_2$ direction, and another moving in the  $e_1$ direction, which will eventually collide.  Their sum will then be large enough to force the reaction term to be negative, which will cause the population density of the species moving in the $e_2$ direction to become negative.

%{\color{blue} MAKE SURE TO ADD REMARK ABOUT HOW TO DO A MORE GENERAL VERSION OF THIS EXAMPLE}

\begin{proof}[Proof of \Cref{prop:kinetic_negative}]
%Recall that $p$ satisfies, for all $(t,x) \in \R_+ \times \R^2$,
%\[
%	p_t + \frac{1}{\sqrt \tau} e \cdot D p = \frac{1}{\tau} \left( \frac{1}{4} \rho - p\right) + \frac{1}{4} \rho(1 - \rho).
%\]

To choose $p_0$, we introduce the (moving) sets
\[\begin{cases}
	X_{1,t} = \{(x_1,x_2) \in \R^2: |x_2| < -x_1+a_{2,\tau}t\},   \quad   X_{2,t} = \{(x_1,x_2) \in \R^2: |x_1| < x_2 + a_{2,\tau}t\}, \quad \text{and} \\[1.5mm]
	%X_{2,t} &= \{(x_1,x_2) \in \R^2: |x_1| < x_2 + a_{n,\tau}t\},\qquad \text{ and}\\
	X_{3,t} = \{(x_1,x_2) \in \R^2: |x_2| < x_1 + a_{2,\tau}t\},
\end{cases}\]
which, when $t=0$,  are disjoint cones with a vertex at the origin in the directions $-e_1$, $e_2$, and $e_1$, respectively,  while, for $t>0$, they move in time in the directions $e_1$, $-e_2$, and $-e_1$, respectively, with  speed $a_{2,\tau}$.
\smallskip

Let $p_0$ be so that 
\[\begin{cases}
	p_0(x,e_2) \equiv 0, \quad  \ p_0(x,e_1) = 4\left(1 - \frac{3}{\tau}\right) \1_{X_{1,0}}(x),\\[1.5mm] 
	%p_0(x,e_1) = 4\left(1 - \frac{3}{\tau}\right) \1_{X_{1,0}}(x),  \quad \text{and} 
	p_0(x,-e_2) = 4\left(1 - \frac{3}{\tau}\right)\1_{X_{2,0}}(x),  \quad \text{and} \quad  %\qquad \text{ and}\\
	p_0(x,-e_1) = 4\left(1 - \frac{3}{\tau}\right)\1_{X_{3,0}}(x).
\end{cases}\]
The constant $4(1-3/\tau)$ is chosen so that $p_0$ is in a quasi-equilibrium,  that is, for any fixed $e \in S^1_d$, the right hand side of~\eqref{eq:discrete_kinetic} is zero at $t=0$. 
%Although the above defined $p_0$ is singular, it is clear Thus, one may easily construct initial data satisfying $p_0 \leq 1$ which, nonetheless, approximate our initial data here after finite time.
\smallskip

A simple computation shows that, in view of the choice  of the sets $X_{i,0}$, $$0\leq \rho(\cdot,0) \leq 1 \ \ \text{in} \ \ \R^2.$$ 
%for all $x \in \R^2$.%  We may assume that $p(x,v,t) \geq 0$ for all $x$, $v$, and $t$ since otherwise the proof is finished.
%\smallskip
To simplify the notation, we introduce 
\[
	\overline p(x,t) := e^{-t/\tau} p(x + e_2 a_{2,\tau} t,e_2,t),
\]
which satisfies %Then $\overline p$ satisfies, for all $t>0$ and $x \in \R^2$,
\begin{equation}\label{eq:kinetic_negative_y_equation}
	\overline p_t = \frac{1}{\tau}\overline \rho + \overline \rho(1-\overline \rho)  \ \  \text{ in } \ \ \R^2 \times S^1_d \times \R_+,
\end{equation}
%\smallskip
with  $\overline \rho(x,t) = \rho(x + e_2 a_{2,\tau}t,t)$.
For future reference note that the right hand side of~\eqref{eq:kinetic_negative_y_equation}  is decreasing in $\overline \rho$ when $\tau > 1$ and $\overline \rho \geq (\tau+1)/2\tau$. 
\smallskip

Fix $\epsilon>0$.  %sufficiently small so that $4(1-\epsilon) \geq 1 + 1/\tau$?????.
It is easy to see, from equation~\eqref{eq:discrete_kinetic},  that there exists $T_\epsilon$, depending only on $\epsilon$, such that, for $t \in (0,T_\epsilon)$,
\begin{equation}\label{chris40}
\begin{split}
%	p_0(x,e_1) = 4(1-\e)\left(1 - \frac{3}{\tau}\right) \1_{X_{1,t}}(x),\quad 
%	p_0(x,-e_2) = 4(1-\e)\left(1 - \frac{3}{\tau}\right)\1_{X_{2,t}}(x),\quad \text{ and}\\[1.5mm]
%	p_0(x,-e_1) = 4(1-\e)\left(1 - \frac{3}{\tau}\right)\1_{X_{3,t}}(x).
	4\left(1 - \frac{3}{\tau}\right)\1_{X_{1,t}}(x) &\geq p(x, e_1,t) \geq 4(1-\epsilon)\left(1 - \frac{3}{\tau}\right)\1_{X_{1,t}}(x),\\
	 4\left(1 - \frac{3}{\tau}\right)\1_{X_{2,t}}(x) &\geq p_0(x,-e_2) \geq 4(1-\e)\left(1 - \frac{3}{\tau}\right)\1_{X_{2,t}}(x), \qquad \text{ and }\\
	4\left(1 - \frac{3}{\tau}\right)\1_{X_{3,t}}(x) &\geq p(x, -e_1,t) \geq 4(1-\epsilon)\left(1 - \frac{3}{\tau}\right)\1_{X_{3,t}}(x).
%	\qquad \text{ and } \qquad
%	p((-1,0),x,t) \geq (1-\epsilon)\1_{(-t/\sqrt{\tau},\infty)\times \R}.
\end{split}
\end{equation}
Fix $\delta \in (0, 2T_\e a_{2,\tau})$.  We point out that the domains of the indicator functions in~\eqref{chris40} are initially disjoint but contain $(0,-\delta)$ when $t>\delta/2a_{2,\tau}$.    Then~\eqref{chris40} yields, for all $t \in (\delta/2a_{2,\tau},T_\e)$,
\begin{equation}\label{chris41}
	\overline \rho((0,-\delta),t) \geq 3(1-\e)\left(1 - \frac{3}{\tau}\right),
\end{equation}
and, hence, for $\e$ small enough and $t > \delta/2a_{2,\tau}$ and since $\tau > 4$, 
\begin{equation}\label{chris101}
	\overline \rho((0,-\delta),t) \geq (\tau+1)/2\tau.
\end{equation}
%if $\e$ is chosen small enough.  Here we use the assumption $\tau > 4$.  We use this inequality~\eqref{chris101} below.
%\smallskip

Using~\eqref{chris101}, the lower bound in~\eqref{chris41}, and the fact that $\tau > 4$  and choosing an even smaller $\e$ in \eqref{eq:kinetic_negative_y_equation}, we obtain, for all $t \in (\delta/2a_{2,\tau},T_\e)$,
%For $t \in (0,T_\epsilon)$, we obtain, via~\eqref{eq:kinetic_negative_y_equation},
\begin{equation}\label{chris300}
\begin{split}
	\overline p_t((0,-\delta),t)
		&= \frac{1}{\tau}\overline \rho((0,-\delta),t) +\overline \rho((0,-\delta),t)(1-\overline \rho((0,-\delta),t))\\
		&\leq \frac{3(1 - \e)}{\tau}\left(1 - \frac{3}{\tau}\right) + 3(1 - \e)\left(1 - \frac{3}{\tau}\right)\left(1 - 3(1-\e) \left(1 - \frac{3}{\tau}\right) \right)\\
		&= 3(1-\e) \left(1 - \frac{3}{\tau}\right) \left(1 + \frac{1}{\tau} - 3(1-\e)\left(1 - \frac{3}{\tau}\right)\right)
		< 0.
\end{split}
\end{equation}
%The last inequality comes from the $\tau > 4$ along with decreasing $\e$, if necessary.  
By a similar, though simpler, computation, $\overline p_t((0,-\delta),t)= 0$ when $t\in [0,\delta/2a_{2,\tau}]$.  Using this,~\eqref{chris300}, and that $\overline p((0,-\delta),0)=0$,  we see that, for all $t \in(\delta/2a_{2,\tau},T_\e)$,
\[\begin{split}
	p\left(\left(a_{2,\tau}t - \delta\right) e_2,e_2,t\right)
		= \overline p((0,-\delta),t)
		< 0,
		%&\qquad< \frac{\delta}{\tau a_{2,\tau}} - 3(1-\e) \left(1 - \frac{3}{\tau}\right) \left(1 + \frac{1}{\tau} - 3(1-\e)\left(1 - \frac{3}{\tau}\right)\right)\left(T_\e - \frac{\delta}{a_{2,\tau}}\right) < 0,
\end{split}\]
and the proof is now complete.
\end{proof}
%
%Fix $a\in(0,T_\epsilon)$ to be determined later.  Notice that $p(t,(0,1),x) \leq a(1-\epsilon)(\epsilon + 1/\tau)/4$ for all $x \in (-a/\sqrt \tau, a/ \sqrt \tau)$.  This is true since time runs between $0$ and $a$ and the right hand side of~\eqref{eq:2d_kinetic} is bounded above by $(1-\epsilon)(\epsilon + 1/\tau)/4$ as this side is monotonically decreasing in $\rho$ and $\rho \geq 1-\epsilon$ everywhere.
%
%For $(t,x) \in (a,T_\epsilon)\times(-a/\sqrt{\tau},a/\sqrt{\tau})$.  Define the function
%\[
%	\overline p(t,x) = - (t- a)(a^2 - \tau x^2) a^{-2} r + a(1-\epsilon)(\epsilon + 1/\tau)/4
%\]
%for $r>0$ to be determined.  We note that $\overline p(t,\pm a/\sqrt{n\tau}) \geq 
%\end{proof}

\appendix

\numberwithin{equation}{section}

\section{Appendix: The integrals~\eqref{eq:Phi} and~\eqref{eq:v_squared}}\label{appendix}

We compute the  three simple integrals used above.

\subsection*{The integral~\eqref{eq:Phi}}
When $n=1$, this computation is simple, since  it reduces to a sum.  On the other hand, when $n > 3$, the expression becomes more complicated.  As such, we omit these cases.  Here, 
for $s > 1$, $w \in S^{n-1}$, and $n =2,3$, we compute the integral
\begin{equation}\label{chris102}
	\Phi(s) = \fint_{S^{n-1}} \frac{dv}{s + v\cdot w}.
\end{equation}
%for $s > 1$ and $w \in S^{n-1}$ and where $n =2,3$.  
%
%When $n=1$, this computation is simple, since  it reduces to a sum.  On the other hand, when $n > 3$, the expression becomes more complicated.  As such, we omit these cases.
%\smallskip
%
When $n=2$, changing variables we may assume that $w = (|w|,0)$.  This, along with using spherical coordinates, yields the integral
\[
	\Phi(s) = \fint_{S^{1}} \frac{dv}{s + v\cdot w}
		= \frac{1}{\pi} \int_0^{\pi} \frac{d\theta}{s+ \cos(\theta)}.
\]
%We used the identity $\cos(2\pi - \theta) = \cos(\theta)$, above.
The substitution $r = \tan(\theta/2)$  then implies %and use the identities $\cos(\theta) = (1- r^2)/(1+r^2)$ and $\cos(\theta/2)^2 = (1+r^2)^{-1}$ to obtain
\[
	\Phi(s)
		= \frac{2}{\pi} \int_0^\infty \frac{dr}{2(1+r^2) + (1-r^2)}
		= \frac{2}{\pi} \int_0^\infty \frac{dr}{s+1 + (s - 1) r^2}
		= \frac{1}{\sqrt{ s^2 - 1}}.
\]
When $n=3$, again we change variables  so that $v\cdot w = v_1|w|$.  Then, using spherical coordinates,~\eqref{chris102} becomes 
\[
	\Phi(s) = \fint_{S^2} \frac{dv}{a + v\cdot w}
		= \frac{1}{2} \int_0^{\pi} \frac{\sin(\theta) d\theta}{a + |w| \cos(\theta)}.
\]
The substitution $r = \tan(\theta/2)$ yields %We make the same substitution as above, noting that $\sin(\theta) = 2r/(1+r^2)$, to obtain
\[\begin{split}
	\Phi(s)
		&= 2 \int_0^{\infty} \frac{r}{(2 + 1) + 2sr^2 +(s-1)r^4}dr
		= \frac{2}{s - 1} \int_0^{\infty} \frac{r}{(r^2 + 1)\left(r^2 + \frac{s + 1}{s - 1}\right)}dr\\
		&= \int_0^\infty \left[\frac{r}{r^2 + 1} - \frac{r}{r^2 +\frac{s+1}{s-1}} \right] dr
		= \frac{1}{2} \log\left( \frac{s + 1}{s-1}\right).
\end{split}\]

\subsection*{The convergence of the integral in~\eqref{chris8} to infinity}

We prove  that the integral term in~\eqref{chris8} tends to infinity as $\mu \to \mu_c^-$ where $\mu_c := (n-1)/2$.  
\smallskip
 
Following the work above, we have, for some normalizing factor $\omega_n$ depending only on  $n$,
\[
	\fint_{S^{n-1}} \frac{dv}{(1 + v_1)^\mu}
		= \frac{1}{\omega_n}\int_{S^{n-1}} \frac{\sin(\theta)^{n-2}}{(1 + \cos(\theta))^\mu}d\theta,
\]
%where $\omega_n$ is a normalizing constant depending only on the dimension $n$.  
and the change of variables $r = \tan(\theta/2)$ yields
\begin{equation}\label{eq:chris2000}
	\fint_{S^{n-1}} \frac{dv}{(1 + v_1)^\mu}
		= \frac{1}{\omega_n}\int_0^\infty \frac{\frac{(2r)^{n-2}}{(1 + r^2)^{n-2}}}{\left(1 + \frac{1-r^2}{1+r^2}\right)^\mu} \frac{2dr}{1+r^2}
		= \frac{2^{n-1 - \mu}}{\omega_n} \int_0^\infty \frac{r^{n-2}}{(1 + r^2)^{n-1 - \mu}}dr.
\end{equation}
This integral is finite only when 
%Looking at the decay of the integrand as $r \to \infty$, it is clear that the integrand is only integrable when 
$\mu < (n-1)/2$.  %From this, it follows that as $\mu \to \mu_c^-$, the integral above tends to infinity.
When $\mu = 1$, the integral is finite if and only if $n > 3$.  Hence, $\Phi(1)< \infty$ if and only if $n > 3$.

\subsection*{The integral~\eqref{eq:v_squared}}

We show here that, for any vector $w \in S^{n-1}$,
\[
	\fint_{S^{n-1}} n(v\cdot w)^2 dv = 1. 
\]
The general result follows by scaling. Moreover, we  consider only the case $n > 2$, since,  for $n=1,2$ the formula is almost immediate.
\smallskip

%In what follows we assume that  $|w| = 1$. The general result follows by scaling.
Changing variables so that the term $v\cdot w= v_1$ gives,  
\[
	\fint_{S^{n-1}} (v\cdot w)^2 dv
		= \fint_{S^{n-1}} v_1^2 dv_1. 
\]
Using  polar coordinates, we obtain
\[\begin{split}
	\fint_{S^{n-1}} (v\cdot w)^2 dv
		&= \frac{1}{\int_0^\pi \sin(\theta)^{n-2} d\theta} \int_0^\pi \sin(\theta)^{n-2} \cos(\theta)^2 d\theta\\
		&= \frac{1}{\int_0^\pi \sin(\theta)^{n-2} d\theta} \int_0^\pi( \sin(\theta)^{n-2} - \sin(\theta)^n) d\theta
		=1 - \frac{\int_0^\pi \sin(\theta)^n d\theta}{\int_0^\pi \sin(\theta)^{n-2} d\theta}.
\end{split}\]
Denoting $I_n = \int_0^\pi \sin(\theta)^n d\theta$, which is sometimes referred to as the Wallis integral, we integrate by parts to obtain the recurrence relation
\[
	I_n = I_{n-2} - \int_0^\pi \sin(\theta)^{n-2} \cos(\theta)^2 d\theta
		= I_{n-2} - \frac{1}{n-1} I_n,
\]
which  implies that $I_n / I_{n-2} = \frac{n-1}{n}$, and, hence, 
\[
	\fint_{S^{n-1}} (v\cdot w)^2 dv
		= 1 - \frac{I_n}{I_{n-2}} = \frac{1}{n}.
\]
%as desired, concluding the computation of the integral~\eqref{eq:v_squared}.

\bibliographystyle{abbrv}
\bibliography{refs}
\end{document}